\documentclass[envcountsame]{llncs}

\usepackage{hyperref}
\usepackage{caption}
\usepackage{tabularx}
\usepackage{algorithm} 
\usepackage{algorithmic} 
\usepackage{graphicx}
\usepackage{lscape}
\usepackage{amsmath}
\usepackage{amssymb}
\usepackage{enumerate}
\usepackage{amsfonts}
\usepackage{mathrsfs}
\usepackage{multirow}
\usepackage{makeidx}
\usepackage{slashed}
\usepackage{color}
\usepackage[utf8]{inputenc}
\usepackage{xypic}
\usepackage{graphicx}
\usepackage{verbatim}

\usepackage{ntheorem}

\usepackage{url}

\usepackage{relsize}

\usepackage{etoolbox}
\patchcmd{\thebibliography}{\chapter*}{\section*}{}{}

\usepackage{ntheorem}
\theoremseparator{.}

\newtheorem*{theorem-non}{Theorem}

\newcommand{\C}{\mathcal{C}}
\newcommand{\R}{\mathbb{R}}

\newcommand{\B}{\mathcal{B}}
\newcommand{\I}{\mathcal{I}}
\newcommand{\M}{\mathcal{M}}
\newcommand{\s}{\oplus_2}

\newcommand{\conv}{\mathrm{conv}}

\newcommand{\calB}{\mathcal{B}}
\newcommand{\calF}{\mathcal{F}}
\newcommand{\rk}{\mbox{rk}}

\newcommand{\stab}{\mathrm{STAB}}
\newcommand{\Hans}{\mathrm{Hans}}

\newtheorem{conj}[theorem]{Conjecture}

\newtheorem{eexample}{Example}

\begin{document}

\title{{On $2$-level polytopes \\ arising in combinatorial settings}\thanks{The third author was partially supported by the Swiss National Science Foundation.}}
\author{Manuel Aprile\inst{1} \and Alfonso Cevallos\inst{2} \and Yuri Faenza\inst{3}}
\institute{DISOPT, \'{E}cole Polytechnique F\'{e}d\'{e}rale de Lausanne, Switzerland.\\ Email: \texttt{manuel.aprile@epfl.ch} \and 
	Department of Mathematics, ETH Zurich, Switzerland.\\ Email: \texttt{alfonso.cevallos@ifor.math.ethz.ch} \and
	IEOR, Columbia University, USA.\\ Email: \texttt{yf2414@columbia.edu}}
\date{}
\maketitle

\begin{abstract}
	$2$-level polytopes naturally appear in several areas of pure and applied mathematics, including combinatorial optimization, polyhedral combinatorics, communication complexity, and statistics. In this paper, we present a study of some $2$-level polytopes arising in combinatorial settings. Our first contribution is proving that $f_0(P)f_{d-1}(P)\leq d2^{d+1}$ for a large collection of families of such polytopes $P$. Here $f_0(P)$ (resp. $f_{d-1}(P)$) is the number of vertices (resp. facets) of $P$, and $d$ is its dimension. Whether this holds for all 2-level polytopes was asked in~\cite{bohn2015enumeration}, and experimental results from~\cite{Fiorini2016} showed it true for $d\leq 7$. The key to most of our proofs is a deeper understanding of the relations among those polytopes and their underlying combinatorial structures. This leads to a number of results that we believe to be of independent interest: a trade-off formula for the number of cliques and stable sets in a graph; a description of stable matching polytopes as affine projections of certain order polytopes; and a linear-size description of the base polytope of matroids that are $2$-level in terms of cuts of an associated tree. 
	

\end{abstract}











\section{Introduction}\label{sec:intro}

Let $P\subseteq \R^d$ be a polytope. We say that $P$ is \emph{$2$-level} if, for each facet $F$ of $P$, all the vertices of $P$ that are not vertices of $F$ lie in the same translate of the affine hull of $F$. Equivalently, $P$ is $2$-level if and only if it has theta-rank $1$~\cite{Gouveia10}, or all its pulling triangulations are unimodular~\cite{Sullivant06}, or it has a slack matrix with entries in $\{0,1\}$~\cite{bohn2015enumeration}. Those last three definitions appeared in papers from the semidefinite programming, statistics, and polyhedral combinatorics communities respectively, showing that $2$-level polytopes naturally arise in many areas of mathematics. 

$2$-level polytopes generalize Birkhoff~\cite{ziegler1995lectures}, Hanner~\cite{Hanner56}, and Hansen polytopes~\cite{Hansen77}, order polytopes and chain polytopes of posets~\cite{Stanley86}, spanning tree polytopes of series-parallel graphs~\cite{Grande14}, stable matching polytopes~\cite{gusfield1989stable}, some min up/down polytopes~\cite{lee2004min}, and stable set polytopes of perfect graphs~\cite{Chvatal75}. A fundamental result in polyhedral combinatorics shows that the linear extension complexity of polytopes from the last class is subexponential in the dimension~\cite{yannakakis1991expressing}. Whether this upper bound can be pushed down to polynomial is a major open problem. For $2$-level polytopes, the situation is even worse: no non-trivial bound is known for their linear extension complexity. On the other hand, $2$-level polytopes admit a ``smallest possible'' semidefinite extension, i.e. of size $d+1$, with $d$ being the dimension of the polytope~\cite{Gouveia10}\footnote{Although $2$-level polytopes are not the only polytopes with this property, they are exactly the class of polytopes for which such an extension can be obtained via a certain hierarchy. See~\cite{Gouveia10} for details.}. Hence, they are prominent candidates for showing the existence of a strong separation between the expressing power of exact semidefinite and linear extensions of polytopes. Interest in $2$-level polytopes is also motivated by their connection to the prominent \emph{log-rank conjecture} in communication complexity~\cite{Lovasz1988lattices}. If this were true, then $2$-level polytopes would admit a linear extension of subexponential size. We defer details on this to the end of the current section.

\begin{figure}[h]
	\includegraphics[width=\textwidth]{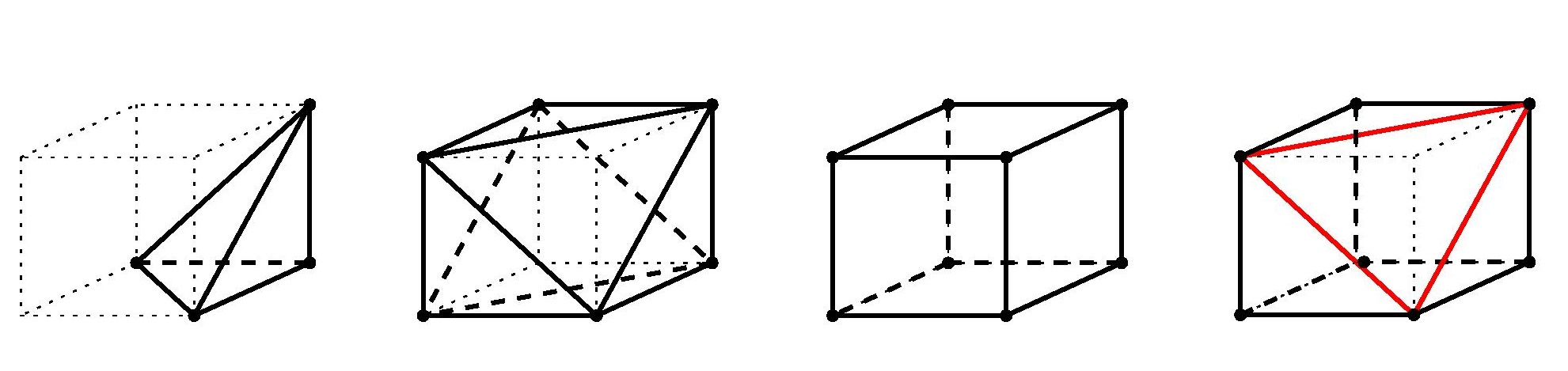}
	\label{pic:2-level}
	\caption{The first three polytopes (the simplex, the cross-polytope and the cube) are 2-level. The fourth one is not 2-level, because of the highlighted facet. }
\end{figure}

Because of their relevance, a solid understanding of $2$-level polytopes would be desirable. Unfortunately, and despite an increasing number of recent studies~\cite{WG17,bohn2015enumeration,Gouveia10,Grande15b,Grande14}, such an understanding has not been obtained yet. We do not have e.g. any decent bound on the number of $d$-dimensional $2$-level polytopes, nor do we have a structural theory of their slack matrices, of the kind that has been developed for totally unimodular matrices; see e.g.~\cite{schrijver1998theory}. 

Still, those works have suggested promising directions of research, either by characterizing specific classes of $2$-level polytopes, or by proving general properties. For instance, in~\cite{Grande14}, by building on Seymour's decomposition theorem for $3$-connected matroids and on the description of uniform matroids in terms of forbidden minors, a characterization of $2$-level polytopes that are base polytopes of matroids is given. On the other hand, it is shown in~\cite{Gouveia10} that each $d$-dimensional $2$-level polytope is affinely isomorphic to a $0/1$ polytope, hence it has at most $2^d$ vertices. Interestingly, the authors of~\cite{Gouveia10} also showed that a $d$-dimensional $2$-level polytope also has at most $2^d$ facets. This makes $2$-level polytopes quite different from ``random'' $0/1$ polytopes, that have $(d/\log d)^{\Theta(d)}$ facets~\cite{barany20010}. Experimental results from~\cite{bohn2015enumeration,Fiorini2016} suggest that this separation could be even stronger: up to $d=7$, the product of the number of facets $f_{d-1}(P)$ and the number of vertices $f_0(P)$ of a $d$-dimensional  $2$-level polytope $P$ does not exceed $d2^{d+1}$. In~\cite{bohn2015enumeration}, it is asked whether this always holds, and in their journal version the question is turned into a conjecture.

\begin{conj}[Vertex/facet trade-off]\label{main-conjecture}
	Let $P$ be a $d$-dimensional $2$-level polytope. Then $$f_0(P)f_{d-1}(P)\leq d2^{d+1}.$$ Moreover, equality is achieved if and only if $P$ is affinely isomorphic to the cross-polytope or the cube.  
\end{conj}

It is immediate to check that the cube and the cross-polytope (its polar) indeed verify $f_0(P)f_{d-1}(P)=d2^{d+1}$. Conjecture~\ref{main-conjecture} has an interesting interpretation as an upper bound on the ``size'' of slack matrices of $2$-level polytopes, since $f_0(P)$ (resp. $f_{d-1}(P)$) is the number of columns (resp. rows) of the (smallest) slack matrix of $P$. Many fundamental results on linear extensions of polytopes are based on properties of their slack matrices. We believe that advancements on Conjecture~\ref{main-conjecture} may lead to precious insights on the structure of (the slack matrices of) $2$-level polytopes, similarly to how progresses on e.g. the outstanding Hirsch~\cite{Santos} and $3^d$ conjectures for centrally symmetric polytopes~\cite{Kalai89} shed some light on our general understanding of polytopes. 

\medskip

\noindent {\bf Our contributions.} The goal of this paper is to present a study of $2$-level polytopes arising from combinatorial settings. Our main results are the following:

\begin{itemize}
	\item[$\bullet$] We give considerable evidence supporting Conjecture~\ref{main-conjecture} by proving it for several classes of $2$-level polytopes arising in combinatorial settings. These include all those cited in the paper so far, plus double order polytopes, and all matroid cycle polytopes that are $2$-level. We moreover show examples of $0/1$ polytopes with a simple structure (including spanning tree and forest polytopes) that are not $2$-level and \emph{do not} satisfy Conjecture~\ref{main-conjecture}. This suggests that, even though there are clearly polytopes that are not $2$-level and satisfy Conjecture~\ref{main-conjecture}, $2$-levelness seem to be the ``appropriate'' hypothesis to prove a general and meaningful result. We also investigate extensions of the conjecture in terms of matrices and systems of linear  inequalities.
	\item[$\bullet$] We establish new properties of many classes of 2-level polytopes, of their underlying combinatorial objects, and of their inter-class connections.  
	These results include: a trade-off formula for the number of stable sets and cliques in a graph; 
	a description of the stable matching polytope as an affine projection of the order polytope of the associated rotation poset; 
	and a compact linear description of $2$-level base polytopes of matroids in terms of cuts of some trees associated to those matroids (notably, our description has linear size in the dimension and  can be written down explicitly in polynomial time).
	These results simplify the algorithmic treatment of some of these polytopes, as well as provide a deeper combinatorial understanding of them.
	At a more philosophical level, these examples suggest that being $2$-level is a very attractive feature for a (combinatorial) polytope, since it seems to imply a well-behaved underlying structure.
\end{itemize}

\medskip

\noindent {\bf Organization of the paper.} We introduce some basic definitions and techniques in Section~\ref{sec:basic}: those are enough to show that Conjecture~\ref{main-conjecture} holds for Birkhoff and Hanner polytopes. 
In Section~\ref{sec:stable-vs-clique}, we first prove an upper bound on the product of the number of stable sets and cliques of a graph (see Theorem~\ref{thr:stables-cliques}). We then prove Conjecture~\ref{main-conjecture} for stable set polytopes of perfect graphs, Hansen polytopes, min up-down polytopes, order, double order and chain polytopes of posets, and stable matching polytopes, by reducing these results to statements on stable sets and cliques of associated graphs, which are also proved in Section~\ref{sec:stable-vs-clique}. Hence, we call all those \emph{graphical} $2$-level polytopes. 
Of particular interest is our observation that stable matching polytopes are affinely equivalent to order polytopes (see Theorem~\ref{marriage}). 
In Section~\ref{sec:matroids}, we give a compact 
description of $2$-level base polytopes of matroids (see Theorem~\ref{thr:matroid}) and a proof of Conjecture~\ref{main-conjecture} for this class (see Lemma~\ref{thr:conj-matroid}). In Section~\ref{sec:cycle-matroid}, we prove the conjecture for the cycle polytopes of certain binary matroids, which generalize all cut polytopes that are $2$-level. In Section~\ref{sec:conclusion}, we investigate possible extensions of the conjecture.

\medskip

%

\noindent {\bf Related work.} We already mentioned the paper~\cite{bohn2015enumeration} that provides an algorithm based on the enumeration of closed sets to list all $2$-level polytopes, as well as papers~\cite{Gouveia10,Grande14,Sullivant06} where equivalent definitions and/or families of $2$-level polytopes are given. In~\cite{Fiorini2016}, the algorithm from~\cite{bohn2015enumeration} is extended and the geometry of 2-level polytopes with some prescribed facet is studied. Among other results, in~\cite{Gouveia10} it is shown that the stable set polytope of a graph $G$ is $2$-level if and only if $G$ is perfect. A characterization of all base polytopes of matroids that are $2$-level is given in~\cite{Grande14}, building on the decomposition theorem for matroids that are not $3$-connected (see e.g.~\cite{oxley2006matroid}). 
A similar characterization of 2-level cycle polytopes of matroids (which generalize cut polytopes) is given in~\cite{gouveia2012new}.

As already pointed out, $2$-level polytopes play an important role in the theory of linear and semidefinite extensions. The (\emph{linear}) \emph{extension complexity} $xc(P)$ of a polytope $P$ has recently imposed itself as a important measure of the complexity of $P$. It is the minimum number of inequalities in a linear description of an extended formulation for $P$. In~\cite{yannakakis1991expressing}, it is shown that the stable set polytope of a perfect graph with $d$ nodes has extension complexity $d^{O(\log d)}$. Whether this bound can be improved or extended to all $2$-level polytopes is unknown. On the other hand, the \emph{semidefinite extension complexity} of a $d$-dimensional $2$-level polytope is $d+1$~\cite{Gouveia10}. $2$-level polytopes are therefore a good candidate to show an exponential separation between the power of exact semidefinite and linear formulations for $0/1$ polytopes. The \emph{log-rank conjecture} aims at understanding the amount of information that needs to be exchanged between two parties in order to compute an input $0/1$ matrix. Exact definitions are not relevant for the scope of this paper, and we refer the interested reader to e.g.~\cite{lovett2014recent}. Here it is enough to note that, if true, this conjecture would imply that the extension complexity of a $d$-dimensional $2$-level polytope is at most $2^{poly \log (d)}$, hence subexponential. 

\medskip

\noindent {\bf About this version.} A preliminary version of this paper appeared in~\cite{ISCOus}, as well as in an earlier version~\cite{aprile2017on} of the present arXiv paper. 
The current full version contains the following additional material: a treatment of min up/down polytopes (and their relation to Hansen polytopes), chain polytopes and double order polytopes of posets, and stable matching polytopes. 
Furthermore, we simplified the proof that 2-level matroid base polytopes satisfy the conjecture, as well as the proof of their complete description in the original space. 
We also significantly extended the discussion on possible generalizations of the conjecture. 
On the other hand, the preliminary version \cite{aprile2017on,ISCOus} also contains a characterization of flacets of the $2$-sum of matroids, as well as an alternative proof of the (known) characterization of $2$-level matroid polytopes (Theorem~\ref{thr:matroid-2-level}).

\section{Basics}\label{sec:basic}

We let $\R_+$ be the set of non-negative real numbers. For a set $S$ and an element $e$, we denote by $A+e$ and $A-e$ the sets $A\cup \{e\}$ and $A\setminus \{e\}$, respectively. For a point $x\in\R^I$, where $I$ is an index set, and a subset $J\subseteq I$, we let $x(J)=\sum_{i\in J} x_i$.

For basic definitions about polytopes and graphs, we refer the reader to~\cite{ziegler1995lectures} and~\cite{diestel2005graph}, respectively. 
The polar of a polytope $P\subseteq \R^d$ is the polyhedron $P^{\triangle}=\{y \in \R^d : x^\intercal y \leq 1, \text{ for all } x\in P\}$. 
It is well known\footnote{It immediately follows from e.g. \cite[Theorem 2.11]{ziegler1995lectures}.} that, if $P\subseteq \R^d$ is a $d$-dimensional polytope with the origin in its interior, then so is $P^{\triangle}$, and one can define a one-to-one mapping between vertices (resp. facets) of $P$ and facets (resp. vertices) of $P^\triangle$. Thus, a polytope as above and its polar will simultaneously satisfy or not satisfy Conjecture~\ref{main-conjecture}. The $d$-dimensional \emph{cube} is $[-1,1]^d$, and the $d$-dimensional \emph{cross-polytope} is its polar. A 0/1 polytope is the convex hull of a subset of the vertices of $\{0,1\}^d$. The following facts will be used many times:
\begin{lemma}\label{lem:2-level-2^d}\cite{Gouveia10}
	Let $P$ be a 2-level polytope of dimension $d$. Then 
	\begin{enumerate}
		\item $f_0(P), f_{d-1}(P)\leq 2^d$.
		\item Any face of $P$ is again a 2-level polytope.
	\end{enumerate}
\end{lemma} 

One of the most common operation with polytopes is the \emph{Cartesian product}. Given two polytopes $P_1\subseteq \R^{d_1}$, $P_2 \subseteq \R^{d_2}$, their Cartesian product is $P_1\times P_2=\{(x,y)\in \R^{d_1+d_2}: x\in P_1, y\in P_2\}$. This operation will be useful to us as it preserves 2-levelness and the bound of Conjecture~\ref{main-conjecture}.


\begin{lemma}\label{lem:cartesian-products-are-good}
	Two polytopes $P_1, P_2$ are 2-level if and only if their Cartesian product $P_1\times P_2$ is 2-level. Moreover, if two 2-level polytopes $P_1$ and $P_2$ satisfy Conjecture~\ref{main-conjecture}, then so does $P_1\times P_2$.
\end{lemma}
\begin{proof}
	The first part follows immediately from the fact that $P_1=\{x: A^{(1)}x\leq b^{(1)}\}$, $P_2=\{y: A^{(2)}y\leq b^{(2)}\}$, then $P_1\times P_2=\{(x,y) : A^{(1)}x\leq b^{(1)}; A^{(2)}y\leq b^{(2)}\}$, and that the vertices of $P_1\times P_2$ are exactly the points $(x,y)$ such that $x$ is a vertex of $P_1$ and $y$ a vertex of $P_2$.
	
	For the second part, let $P= P_1 \times P_2$, $d_1=d(P_1)$, $d_2=d(P_2)$. Then it is well known that $d(P)=d_1 + d_2$, $f_0(P)=f_0(P_1)f_0(P_2)$, and $f_{d-1}(P)=f_{d_1-1}(P_1)+f_{d_2-1}(P_2)$. We conclude
	\begin{eqnarray*} f_0(P)f_{d-1}(P) & = & f_0(P_1)f_{d_1-1}(P_1)f_0(P_2) + f_0(P_2)f_{d_2-1}(P_2)f_0(P_1)\\
		& \leq & d_1 2^{d_1+d_2+1} + d_2 2^{d_1+d_2+1}\\
		& = & \,d(P) 2^{d(P)+1},\end{eqnarray*}
	
	\noindent where the inequality follows by induction and from Lemma~\ref{lem:2-level-2^d}. Suppose now that $P$ satisfies the bound with equality. Then, for $i=1,2$, $P_i$ also satisfies the bound with equality and $f_0(P_i)= 2^{d(P_i)}$, which means that $P_i$ is a $d_i$-dimensional cube. Then $P$ is a $d$-dimensional cube.\end{proof}

\subsection{Hanner and Birkhoff polytopes}
We start off with two easy examples. \emph{Hanner} polytopes~\cite{Hanner56} are defined as the smallest family that contains the $[-1,1]$ segment of dimension 1, and is closed under taking polars and Cartesian products. 
That they verify the conjecture immediately follows from Lemma~\ref{lem:cartesian-products-are-good} and from the discussion on polars earlier in Section~\ref{sec:basic}. 
The \emph{Birkhoff} polytope $B_n\subset \mathbb{R}^{n^2}$ is the convex hull of all $n \times n$ permutation matrices (see e.g.~\cite{ziegler1995lectures}). 
For $n=2$, the polytope $B_2$ is affinely isomorphic to the Hanner polytope of dimension $1$. 
For $n \geq 3$,  $B_n$ is known~\cite{ziegler1995lectures} to be 2-level and to have exactly $n!$ vertices, $n^2$ facets, and dimension $(n-1)^2$; 
it can be checked from these numbers that the conjectured bound holds and is loose. We conclude the following.

\begin{lemma}
	Hanner and Birkhoff polytopes satisfy Conjecture~\ref{main-conjecture}.
\end{lemma}





\section{Graphical 2-Level Polytopes}\label{sec:stable-vs-clique}

We present a general result on the number of cliques and stable sets of a graph. Proofs of all theorems from the current section will be based on it.

\begin{theorem}[Stable set/clique trade-off]\label{thr:stables-cliques}
	Let $G=(V,E)$ be a graph on $n$ vertices, $\C$ its family of non-empty cliques, and $\cal S$ its family of non-empty stable sets. Then 
	$$|\C||{\cal S}|\leq n(2^n-1).$$ 
	Moreover, equality is achieved if and only if $G$ or its complement is a clique. 
\end{theorem}

\begin{proof}
	Consider the function $f:\C \times {\cal S}\rightarrow 2^V$, where $f(C,S)=C\cup S$. For a set $W\subset V$, we bound the size of its pre-image $f^{-1}(W)$. If $W$ is a singleton, the only pair in its pre-image is $(W,W)$. For $|W|\geq 2$, we claim that $|f^{-1}(W)|\leq 2|W|.$
	
	There are at most $|W|$ intersecting pairs $(C,S)$ in $f^{-1}(W)$. This is because the intersection must be a single element, $C\cap S=\{v\}$, and once it is fixed every element adjacent to $v$ must be in $C$, and every other element must be in $S$.
	
	There are also at most $|W|$ disjoint pairs in $f^{-1}(W)$, as we prove now. Fix one such disjoint pair $(C,S)$, and notice that both $C$ and $S$ are non-empty proper subsets of $W$. All other disjoint pairs $(C',S')$ are of the form $C' = C \setminus A \cup B$ and $S' = S \setminus B \cup A$, where $A \subseteq C$, $B \subseteq S$, and $|A|,|B| \leq 1$. Let $X$ (resp. $Y$) denote the set formed by the vertices of $C$ (resp. $S$) that are anticomplete to $S$ (resp. complete to $C$). Clearly, either $X$ or $Y$ is empty. We settle the case $Y=\emptyset$, the other being similar. In this case $\emptyset\neq A\subseteq X$, so $X\neq \emptyset$. If $X = \{v\}$, then $A=\{v\}$ and we have $|S|+1$ choices for $B$, with $B=\emptyset$ possible only if $|C|\geq 2$, because we cannot have $C'=\emptyset$. This gives at most $1 + |S| + |C| - 1 \leq |W|$ disjoint pairs $(C',S')$ in $f^{-1}(W)$. Otherwise, $|X|\geq 2$ forces $B=\emptyset$, and the number of such pairs is at most $1 + |X| \leq 1 + |C| \leq |W|$.
	
	We conclude that $|f^{-1}(W)|\leq 2|W|$, or one less if $W$ is a singleton. Thus
	$$|\C \times {\cal S}| \leq \sum_{k=0}^n 2k {n\choose k} - n = n 2^n - n,$$
	where the (known) fact $\sum_{k=0}^n 2k {n \choose k}= n 2^{n}$ holds since
	
	$$n 2^n =\sum_{k=0}^n (k+(n-k)) {n \choose k} =  \sum_{k=0}^n k {n \choose k} + (n-k) {n \choose n-k}= 2  \sum_{k=0}^n k {n \choose k}.$$
	
	The bound is clearly tight for $G=K_n$ and $G=\overline{K_n}$. For any other graph, there is a subset $W$ of 3 vertices that induces 1 or 2 edges. In both cases, $|f^{-1}(W)|=5<2|W|$, hence the bound is loose.
\end{proof}

\begin{corollary}\label{cor:trade-off}
	Let $G$, $\C$ and $\cal S$ be as in Theorem~\ref{thr:stables-cliques}, and $\C'=\C\cup \{\emptyset\}$ and ${\cal S'}={\cal S} \cup \{\emptyset\}$ be the families of (possibly empty) cliques and stable sets of $G$, respectively. Then 
	$$|\C'||{\cal S}'|\leq (n+1)2^n,$$ 
	and equality is achieved if and only if $G$ or its complement is a clique.
\end{corollary}

\begin{proof}
	We apply the previous inequality to obtain
	$$\begin{array}{lll} |\C'||{\cal S}'|&=&(|\C|+1)(|{\cal S}|+1)=|\C||{\cal S}|+(|\C|+|{\cal S}'|)\\
	&\leq & n(2^n-1)+(|\C\cup {\cal S'}|+|\C\cap{\cal S}'|)\\
	&\leq & n(2^n-1)+(2^n + n)=(n+1)2^n.
	\end{array}$$
	Clearly the inequality is tight whenever $G$ or its complement is a clique, and from Theorem~\ref{thr:stables-cliques}, we know that it is loose otherwise.
\end{proof}

\subsection{Stable set polytopes of perfect graphs}\label{sec:stab(G)}

For a graph $G=(V,E)$, its stable set polytope $\stab(G)$ is the convex hull of the characteristic vectors of all stable sets in $G$. It is known that $\stab(G)$ is $2$-level if and only if $G$ is a \emph{perfect graph}~\cite{Gouveia10}, or equivalently~\cite{Chvatal75} if and only if
$$\stab(G)=\{x\in\R_+^V: x(C)\leq 1 \text{ for all maximal cliques $C$ of $G$}\}.$$

\begin{lemma}\label{thr:perfect}
	Stable set polytopes of perfect graphs satisfy Conjecture~\ref{main-conjecture}.
\end{lemma}
\begin{proof} 
	For a perfect graph $G=(V,E)$ on $d$ vertices, the polytope $\stab(G)$ is $d$-dimensional. If we define $\C$, $\C'$ and $\cal S'$ as in Corollary~\ref{cor:trade-off}, then the number of vertices in $\stab(G)$ is at most $|{\cal S'}|$. There are at most $d$ non-negativity constraints, and at most $|\C|=|\C'|-1$ clique constraints, so the number of facets in $\stab(G)$ is at most $|\C'|+d-1$. Hence
	$$\begin{array}{lll}f_0(\stab(G))f_{d-1}(\stab(G)) & \leq & (|\C'|+d-1)|{\cal S'}|\\
	& = & |\C'||{\cal S'}|+(d-1)|{\cal S'}| \\
	& \leq & (d+1)2^d+(d-1)2^d=d2^{d+1},\end{array}$$
	where we used Corollary~\ref{cor:trade-off} and the trivial inequality $|{\cal S'}|\leq 2^d$. We see that the conjectured inequality is satisfied, and is tight only in the trivial cases $d=1$ or $|{\cal S'}|= 2^d$. In the latter case, $G$ has no edges and $\stab(G)$ is affinely isomorphic to the cube.
\end{proof}

\subsection{Hansen polytopes}\label{sec:Hans}

Given a $(d-1)$-dimensional polytope $P$, the \emph{twisted prism} of $P$ is the $d$-dimensional polytope defined as the convex hull of $\{(x,1): x\in P\}$ and $\{(-x,-1):x\in P\}$. For a perfect graph $G$ with $d-1$ vertices, its \emph{Hansen} polytope~\cite{Hansen77}, $\Hans(G)$, is defined as the twisted prism of $\stab(G)$. Hansen polytopes are 2-level and centrally symmetric, see e.g.~\cite{bohn2015enumeration}. 
\begin{lemma}\label{thr:Hansen}
	Hansen polytopes satisfy Conjecture~\ref{main-conjecture}.
\end{lemma}

\begin{proof}
	Let $G=(V,E)$ be a perfect graph on $d-1$ vertices, and let $\C'$ and ${\cal S}'$ be as in Corollary~\ref{cor:trade-off}. Then $\Hans(G)$ has $2|{\cal S'}|$ vertices (from the definition), and $2|\C'|$ facets (see e.g.~\cite{Hansen77}). Using again Corollary~\ref{cor:trade-off}, we get
	$$f_0(\Hans(G))f_{d-1}(\Hans(G))  = 4|{\cal S'}||\C'| \leq 4d2^{d-1} = d2^{d+1}.$$
	The inequality is tight only if $G$ is either a clique or an anti-clique. The Hansen polytopes of these graphs are affinely equivalent to the cross-polytope and cube, respectively. \end{proof}

\subsection{Min up/down polytopes}

Fix two integers $0<l<d$. For a 0/1 vector $x\in \{0,1\}^d$ and index $1\leq i\leq d-1$, we call $i$ a \emph{switch index} of $x$ if $x_i\neq x_{i+1}$. Vector $x$ satisfies the \emph{min up/down constraint} (with parameter $l$) if for any two switch indices $i<j$ of $x$, we have $j-i\geq l$. In other words, when $x$ is seen as a bit-string then it consists of blocks of 0's and 1's each of length at least $\ell$ (except possibly for the first and last blocks). The \emph{min up/down polytope} $P_d(l)$ is defined as the convex hull of all 0/1 vectors in $\mathbb{R}^d$ satisfying the min up/down constraint with parameter $l$. Those polytopes have been introduced in~\cite{lee2004min} in the context of discrete planning problems with machines that have a physical constraint on the frequency of switches between the operating and not operating states.\footnote{The more general definition given in~\cite{lee2004min} considers two parameters $\ell_1$ and $\ell_2$, which respectively restrict the minimum lengths of the blocks of $0$'s and $1$'s in valid vertices. The resulting polytope is 2-level precisely when $\ell_1=\ell_2$, thus in this section we restrict our attention to this case. General (non-2-level) min up/down polytopes do not satisfy Conjecture~\ref{main-conjecture}; see Example~\ref{ex:minupdown}.} In~\cite[Theorem 4]{lee2004min}, the following characterization of the facet-defining inequalities of $P_d(l)$ is given.

\begin{theorem}\label{lem:up/down}
	Let $I\subset [d]$ be an index subset with elements $1\leq i_1<i_2<\cdots<i_k\leq d$, such that a) $k=|I|$ is odd and b) $i_k - i_1 \leq l$. Then, the two inequalities $0\leq \sum_{j=1}^k (-1)^{j-1} x_{i_j} \leq 1$ are facet-defining for $P_d(l)$. Moreover, each facet-defining inequality in $P_d(l)$ can be obtained in this way. 
\end{theorem}

It is clear from this result that $P_d(l)$ is a 2-level polytope. Indeed, if all vertices of a polytope have  $0/1$ coordinates and all facet-defining inequalities can be written as $0\leq c^\intercal x\leq 1$ for integral vectors $c$, then the polytope is $2$-level. 

\begin{lemma}\label{thr:up/down}
	2-level min up/down polytopes satisfy Conjecture~\ref{main-conjecture}.
\end{lemma}

\begin{proof}
	Consider the 2-level min up/down polytope $P_d(l)$, for integers $0<l<d$. $P_d(l)$ is full dimensional, hence it has dimension $d$. 
	Define the graph $G([d-1],E)$, where $\{i,j\}\in E$ whenever $|j-i|\leq l-1$, and let $\C'$ and ${\cal S}'$ be as in Corollary~\ref{cor:trade-off}. 
	We delay for a moment the proof of the following facts: a) $f_0(P_d(l))= 2|{\cal S}'|$; and b) $f_{d-1}(P_d(l)) = 2|\C'|$. 
	We obtain: $$f_0(P_d(l))f_{d-1}(P_d(l))  = 4|{\cal S'}||\C'|.$$ 
	This is the same identity that appears in the proof of Proposition~\ref{thr:Hansen}, hence in a similar fashion we conclude that the conjectured inequality is satisfied, and it is tight only if $G$ is either a clique or an anti-clique. 
	These cases correspond to $l=d-1$ and $l=1$, respectively, and it can be checked that $P_d(l)$ is then affinely equivalent to the cross-polytope or the cube.
	
	Proof of fact a). For a vector $x \in \{0,1\}^d$, let $I_x\subseteq [d-1]$ be its set of switch indices. Then $x$ is (a vertex) in $P_d(l)$ iff $I_x$ is a stable set in $G$. Moreover, if two vertices $x,y\in P_d(l)$ have exactly the same switch indices, then either $x=y$ or $x+y= {\bf 1}$ (the all-ones vector). Hence, there is a mapping from the set of vertices of $P_d(l)$ to ${\cal S}'$, where each pre-image contains 2 elements. This proves the claim.
	
	Proof of fact b). Let ${\cal{I}}\subseteq 2^{[d]}$ be the collection of all index sets $I\subseteq [d]$ satisfying the properties of Lemma~\ref{lem:up/down}. The lemma asserts that $f_{d-1}(P_d(l))=2|\cal{I}|$. To complete the proof, we present a bijection from $\cal I$ to $\C'$. For $I\subset [d]$ in $\cal I$, let $i$ be the lowest index in $I$, let $j=\min\{i+l,d\}$, and define $I'=I\setminus \{j\}$. $I'$ is a clique in $G$. We conclude the proof by showing that the mapping can be inverted, hence it is bijective. Recall that $G$ has nodes indexed from $1$ to $d-1$. For $I'\in \C'$, if $|I'|$ is odd, let $I=I'$; if $I'=\emptyset$, let $I=\{d\}$; otherwise, let $i$ be the lowest index in $I$ and $j=\min\{i+l,d\}$, and define $I=I'\cup\{j\}$. Clearly, in all cases $I \in {\cal I}$, and the preimages of two even cliques or two odd cliques are distinct. Now pick an even clique $I'$. If $I'=\emptyset$, then $I=\{d\}$ is not the preimage of an odd clique. If $I'\neq \emptyset$ and $i+l  < d$, then $I$ is not a clique of $G$, hence, in particular, it cannot be an odd clique. If $d \leq i+l$, then $d \in I$, and the latter never occurs for odd cliques.
\end{proof}

We remark that the graph $G=G_{d,l}$ defined in the proof of Proposition~\ref{thr:up/down} is perfect. 
Therefore, in this proof we exhibit for each min up/down polytope $P_d(l)$ a corresponding Hansen polytope $\Hans(G_{d,l})$ with equal dimension, number of vertices, and number of facets as $P_d(l)$. 
It is then natural to wonder if these two polytopes are combinatorially equivalent, or more generally, if min up/down polytopes are just a subclass of Hansen polytopes 
(after all, both classes are 2-level and centrally symmetric up to translation). 
This turns out not to be the case.

\begin{proposition}\label{lem:notHansen}
	The min up/down polytope with parameters $d=8$ and $l=2$ is not combinatorially equivalent to any Hansen polytope.
\end{proposition}
\begin{proof}
	It can be checked computationally that the min up/down polytope $P_8(2)$ is of dimension 8 and contains 68 vertices, 28 facets, and 604 edges (see Appendix~\ref{app:notHansen} for details on the computation). 
	The corresponding perfect graph assigned to it in the proof of Proposition~\ref{thr:up/down} is $P_7$, the path on 7 nodes, 
	and it can be checked as well that its Hansen polytope, $\Hans(P_7)$, is of dimension 8 and contains 68 vertices, 28 facets, and 622 edges (see Appendix~\ref{app:notHansen}). 
	This last number proves that the two polytopes are not combinatorially equivalent.
	
	It remains to show that there is no other perfect graph $G$, for which $\Hans(G)$ is equivalent to $P_8(2)$. 
	Assume by contradiction that there is such a graph $G$, with $n$ nodes and $m$ edges, and let $\C'$ and $\cal S'$ be as in Corollary~\ref{cor:trade-off}.
	From the information we have on $P_8(2)$, and from the proof of Proposition~\ref{thr:Hansen}, it follows that $n=7$, $|\C'|=14$ and $|{\cal S}'|=34$. 
	Notice also that the bound $|\C'|\geq m+n+1$ gives $m\leq 6$.
	Suppose first that $G$ is connected; then the bound on $m$ implies that $G$ is a tree. 
	There is extensive bibliography on the number of stable sets on trees, and it particular it is known~\cite{prodinger1982fibonacci} that $|{\cal S}'|\geq F_{n+2}$ (where $F_i$ is the $i$-th Fibonacci number), 
	and that this bound is tight only in the case of a path. As this bound is tight for $G$, we conclude that $G=P_7$, a case already considered above.
	
	Now suppose that $G$ is not connected. Then the number $|\cal S'|$ of stable sets is equal to the product of the corresponding numbers for each connected component. 
	As $|{\cal S}'|=34$ factors into $2\cdot 17$, $G$ must be composed precisely of two components: an isolated node, and a connected graph $G'$ with $|{\cal S'}_{G'}|=17$ stable sets, $n'=6$ nodes, and $m$ edges, with $5 \leq m\leq 6$. 
	Now, $G'$ cannot be a tree, as in that case $G$ would only have $|\C'|=13$ cliques. 
	Therefore, $G'$ must be a \emph{unicyclic graph}, i.e., a tree with an additional edge. 
	There are also results on the number of stable sets on uniclyclic graphs; in particular, it is known \cite[Thm. 9]{wagner2010maxima} that $|{\cal S'}_{G'}|\geq F_{n'+1}+F_{n'-1}$. 
	This leads to the inequality $17\geq 13+5$, which is a contradiction. This completes the proof.
\end{proof}

\subsection{Polytopes coming from posets}
\label{sec:order}

Consider a poset $P$, with order relation $\preceq$. 
Its associated \emph{order polytope} is 
\begin{equation}\label{poly:order}
	\mathcal{O}(P) = \{ x \in [0,1]^P : \; x_i\geq x_j \; \text{ whenever }  i\preceq j\}, 
\end{equation}
and its \emph{chain polytope} is
\begin{equation}\label{poly:chain}
	\mathcal{C}(P) = \{ x \in \R_+^P : \; \textstyle{\sum_{i\in I}} x_i\leq 1 \; \text{ for each maximal chain } I\subseteq P\}, 
\end{equation}
where we recall that a subset $I\subseteq P$ is a \emph{chain} if every pair of elements in it is comparable. 
Similarly, $I\subseteq P$ is an \emph{anti-chain} if no pair in it is comparable, and it is a \emph{closed set} if $j\in I$ and $i\preceq j$ imply $i\in I$. 
There is a well-known one-to-one correspondence between the closed sets and the anti-chains of a poset (the bijection maps a closed set to the subset formed by its maximal elements, which is an anti-chain). 
Stanley~\cite{Stanley86} gives the following characterization of vertices of these two polytopes. 

\begin{theorem}[\cite{Stanley86}]\label{upsets}
	The vertices of $\mathcal{O}(P)$ are the characteristic vectors of closed sets of $P$, and the vertices of $\mathcal{C}(P)$ are the characteristic vectors of the anti-chains of $P$. 
	In particular, $\mathcal{O}(P)$ and $\mathcal{C}(P)$ have an equal number of vertices. 
\end{theorem}

From this result it is clear that the order polytope $\mathcal{O}(P)$ is a 2-level polytope because, as argued before, it is a sufficient condition that  
all vertices have $0/1$ coordinates and all facet-defining inequalities can be written as $0\leq c^\intercal x\leq 1$ for integral vectors $c$.
Let us now analyze the chain polytope $\mathcal{C}(P)$. 
Define the comparability graph of $P$ as $G_P([d],E)$, with $\{i,j\}\in E$ whenever $i\preceq j$ or $j\preceq i$. 
It is then easy to see that cliques and stable sets of this graph correspond precisely to chains and anti-chains of $P$, respectively. %
But as comparability graphs are perfect (see e.g.~\cite{chvatal1984perfectly}), 
it follows that $\mathcal{C}(P)$ is equal to the stable set polytope of $G_P$, and hence it is 2-level and satisfies Conjecture~\ref{main-conjecture} by Lemma~\ref{thr:perfect}.

The order and chain polytopes of $P$ in general do not have the same number of facets. 
There is, however, a known relation between these numbers, that immediately gives us our desired bound.

\begin{lemma}[\cite{hibi2012unimodular}]\label{hibi}
	The number of facets of $\mathcal{O}(P)$ is less than or equal to the number of facets of $\mathcal{C}(P)$.
\end{lemma}

\begin{lemma}\label{thr:order}
	Order polytopes and chain polytopes satisfy Conjecture~\ref{main-conjecture}.
\end{lemma}

\begin{proof} 
	Given a poset $P$ on $d$ elements, it is easy to see that both ${\cal O}(P)$ and ${\cal C}(P)$ are full dimensional, hence both have dimension $d$. 
	The proof for ${\cal C}(P)$ is already given in the lines above. 
	The claimed bound for ${\cal O}(P)$ now easily follows from the bounds stated in Lemmas~\ref{upsets} and~\ref{hibi}. 
	If this bound is tight for ${\cal O}(P)$, then it must also be tight for ${\cal C}(P)= \stab(G_P)$; 
	this implies by Lemma~\ref{thr:perfect} that $G_P$ has no edges, so $P$ is the trivial poset and ${\cal O}(P)$ is the cube.
\end{proof}

To conclude the section, we mention a class of polytopes defined from double posets, which was studied in~\cite{chappell2017two}. A double poset is a triple $(P,\preceq_+, \preceq_-)$, where $\preceq_+$ and $\preceq_-$ are two partial orders on $P$. The double order polytope is defined as $${\cal O}(P,\preceq_+, \preceq_-)=\conv\{(2{\cal O}(P_+) \times \{1\}) \cup (-2{\cal O}(P_-) \times \{-1\})\},$$
where $P_+$ is the poset relative to $\preceq_+$, and similarly for $P_-$. A double poset is said to be \emph{compatible} if $\preceq_+, \preceq_-$ have a common linear extension (i.e. they can be extended to the same total order). In~\cite{chappell2017two} it is proved that, if $(P,\preceq_+, \preceq_-)$ is compatible, then ${\cal O}(P,\preceq_+, \preceq_-)$ is 2-level if and only if $\preceq_+= \preceq_-$ and that in this case the number of its facets is twice the number of chains of $(P,\preceq_+)$. This leads to the following:

\begin{lemma}\label{thr:doubleorder}
	For any poset $(P,\preceq)$, the double order polytope ${\cal O}(P,\preceq, \preceq)$ satisfies Conjecture~\ref{main-conjecture}.
\end{lemma}
\begin{proof} Let $|P|=d$. From the definition, it is clear that ${\cal O}(P,\preceq, \preceq)$ has dimension $d+1$ and twice as many vertices as ${\cal O} (P)$. Let $A, C$ be the sets of anti-chains and chains of $P$, respectively. Using Lemma~\ref{upsets}, and the result in~\cite{chappell2017two}, we have that ${\cal O}(P,\preceq, \preceq)$ has $2|A|$ vertices and $2|C|$ facets. Now, we remark that Corollary~\ref{cor:trade-off} applied to the comparability graph of $P$ implies that $|A|\cdot |C|\leq (d+1)2^d$, this being tight only if $P$ itself is a chain or an anti-chain. The thesis follows immediately. 
\end{proof}


\subsection{Stable matching polytopes}

An instance of the stable matching (or stable marriage) problem, in its most classical version, is defined by a complete bipartite graph $G(M\cup W, E)$ with $n=|M|=|W|$, together with a list $(<_v)_{v\in M\cup W}$,
where for each vertex $v$, $<_v$ is a strict linear order over $v$'s neighbors. 
The traditional context of the problem is that there is a set $M$ of men and a set $W$ of women, where each individual wishes to marry a member of the opposite set, 
and has a list of strict preferences (for instance, $m<_w m'$ means that $w$ prefers $m'$ over $m$).
A \emph{stable marriage} is a perfect matching $\mu$ in $G$ with the property that there is no un-matched pair where both individuals prefer each other over their partners; 
more precisely, if $\mu(v)$ represents $v$'s partner in matching $\mu$, then $\mu$ is stable if and only if
\[
\forall mw\in E\setminus \mu, \text{ either } m<_w \mu(w) \text{ or }  w<_m \mu(m).
\]

Let $\cal{M}$ be the set of stable matchings of this instance.  
The \emph{stable matching polytope} $S({\cal M})$ is the convex hull of the characteristic vectors of all stable matchings in $\cal{M}$. 
As every instance has at least one stable matching~\cite{gale1962college}, $S({\cal M})$ is a non-empty subset of $[0,1]^E$. 
Furthermore, it is known~\cite{roth1993stable} that $S({\cal M})$ can be described as
\[
\Big\{ x\in\R^E_{\geq 0}: \ x(\delta(v))\leq 1 \ \forall v\in V, \ x_{mw}+\sum_{m'>_{w} m} x_{m'w} + \sum_{w'>_{m} w} x_{mw'} \geq 1 \ \forall mw \in E \Big\}.
\]
From this description, it is evident that $S({\cal M})$ is a 2-level polytope, because all vertices have 0/1 coordinates, and all inequalities are of the form 
$\alpha\leq c^\intercal x \leq \alpha+1$ for some integral vector $c$ and integer $\alpha$.%
\footnote{To visualize this, notice that the above-mentioned description is equivalent to $S({\cal M}) = \{x\in\R^E:$ 
	$0\leq x_{mw}\leq 1 \text{ and } 1\leq x_{mw}+\sum_{m'>_{w} m} x_{m'w} + \sum_{w'>_{m} w} x_{mw'} \leq 2 \text{ for each } mw\in E,  \text{ and }$ 
	$0\leq x(\delta(v))\leq 1 \text{ for each } v\in V \} $. } 
Our strategy is to prove that the stable matching polytope is affinely equivalent to an order polytope, and hence satisfies Conjecture~\ref{main-conjecture} by Proposition~\ref{thr:order}. 
To this end, we first present some necessary notation and results. We then discuss how our results relates to known facts from the literature.

For a pair of stable matchings $\mu, \mu'$ in $\cal{M}$, the relation $\mu \preceq \mu'$ signifies that every woman is at least as happy with $\mu'$ than with $\mu$, 
i.e., for each $w\in W$, either $\mu(w)<_w \mu'(w)$ or $\mu(w)= \mu'(w)$. 
This relation makes $\cal M$ a distributive lattice; see~\cite{knuth1976mariages}. 
We denote by $\mu_0$ and $\mu_z$ respectively the (unique) minimum and maximum in this lattice. 
Further, the ordered pair $(\mu,\mu')$ of distinct stable matchings is a \emph{covering pair} if $\mu \preceq \mu'$ and there is no other $\mu''\in{\cal M}$ such that $\mu\preceq \mu''\preceq \mu'$. 
The lattice structure of $\cal M$ can be represented by its \emph{Hasse diagram}, which is the directed graph $H({\cal M},A)$, where $A$ is the set of all covering pairs. 

The \emph{rotation} generated by a covering pair $(\mu, \mu')\in A$ is defined as $\rho=(\rho^-, \rho^+)$, where $\rho^- = \mu\setminus \mu'$ and $\rho^+ = \mu' \setminus \mu$. 
We refer to sets $\rho^-$ and $\rho^+$ respectively as the tail and the head of rotation $\rho$.%
\footnote{This is not the standard definition of rotation found in the literature, but can be seen to be equivalent by \cite[Thm. 6]{gusfield1987three}. 
	(Our notation is also different, with the traditional notation being as follows. 
	If $\rho=(\rho^-, \rho^+)$ is generated by $(\mu, \mu')$, then $\rho$ is said to be \emph{exposed} in $\mu$; 
	$\mu'$ is said to be obtained from $\mu$ after \emph{eliminating} $\rho$ from it, and denoted by $\mu/ \rho$;
	each edge in $\rho^-$ is \emph{eliminated} by $\rho$, and each edge in $\rho^+$ is \emph{produced} by $\rho$.)}
Let $\Pi$ be the set of all rotations generated by covering pairs in $A$, and notice that more than one covering pair may generate the same rotation in $\Pi$. 
For a pair of rotations $\rho, \rho'$ in $\Pi$, we say that $\rho$ \emph{precedes} $\rho'$,  
if in any $\mu_0-\mu_z$ path $P$ in the Hasse diagram $H$, any arc generating $\rho$ precedes any arc generating $\rho'$.%
\footnote{Again, this is not the standard definition of the precedence relation, but can be seen to be equivalent by \cite[Thm. 3.2.1]{gusfield1989stable}.} 
This precedence relation defines a poset structure over $\Pi$~\cite{irving1986complexity}. 
We now enumerate some properties of the rotation poset $\Pi$.

\begin{lemma}\label{rotation}
	Let $\Pi$ be the rotation poset associated to ${\cal M}$.
	\begin{enumerate}
		\item \cite[Thm. 2.5.4]{gusfield1989stable} For each $\mu \in {\cal M}$, there is a subset $\Pi(\mu)\subseteq \Pi$ such that, 
		for each $\mu_0-\mu$ path $P$ in $H$, the set of rotations generated by arcs in $P$ is precisely $\Pi(\mu)$, with each rotation in it generated exactly once.
		\item \cite[Thm. 2.5.7]{gusfield1989stable} For each $\mu\in {\cal M}$, $\Pi(\mu)$ is a closed set of the rotation poset $\Pi$,  
		and this mapping defines a bijection between ${\cal M}$ and the closed sets in $\Pi$.  
	\end{enumerate}
\end{lemma}

The following proposition was observed in~\cite{EsFa}. For completeness, we give a proof in Appendix~\ref{app:linearindep}. 

\begin{proposition}\label{lem:linearindep}
	Vector family $\big \{\chi^{\rho^+} - \chi^{\rho^-}\big\}_{\rho \in \Pi}$ is linearly independent in $\R^E$.
\end{proposition}

\begin{theorem}\label{marriage} 
	Given a lattice ${\cal M}$ of stable matchings, with associated rotation poset $\Pi$, 
	the stable matching polytope $S({\cal M})$ is affinely equivalent to the order polytope ${\mathcal O}(\Pi)$. 
	More precisely, if $\mu_0$ is the minimal element in ${\cal M}$, then
	\begin{align*}
		S({\cal M})= \chi^{\mu_0} + A \cdot {\mathcal O}(\Pi),
	\end{align*} where $A\in \R^{E \times \Pi}$ is the matrix with columns $A^\rho=\chi^{\rho^+}-\chi^{\rho^-}$ for each $\rho \in \Pi$.  
\end{theorem}

\begin{proof}
	Let $Q$ be the polytope on the right-hand side of the claimed identity. 
	$Q$ is clearly an affine projection of ${\mathcal O}(\Pi)$ into $\R^E$. 
	Further, the affine dimension of $Q$ is equal to that of ${\mathcal O}(\Pi)$, by Lemma~\ref{lem:linearindep}. 
	Hence, $Q$ is affinely equivalent to ${\mathcal O}(\Pi)$. 
	
	It remains to show that $S({\cal M})=Q$, which we do by proving that the collection of vertices of these polytopes coincide. 
	Recall from Lemma~\ref{upsets} that the vertices of ${\mathcal O}(\Pi)$ are precisely the characteristic vectors of the closed sets in $\Pi$, 
	and that these closed sets are in one-to-one correspondence to the stable matchings in $\cal M$, by Lemma~\ref{rotation}~(2). We thus obtain that the vertices of $Q$ are 
	$ \big\{ \chi^{\mu_0} + \sum_{\rho \in \Pi(\mu)} (\chi^{\rho^+} - \chi^{\rho^-}) \big\}_{\mu\in{\cal M}}$.
	
	Finally, we prove that $\chi^\mu=\chi^{\mu_0} + \sum_{\rho \in \Pi(\mu)} (\chi^{\rho^+} - \chi^{\rho^-})$ for each stable matching $\mu$. 
	Fix $\mu\in{\cal M}$, and fix a $\mu_0-\mu$ path $P$ in $H$: this defines a chain of stable matchings $\mu_0\preceq \mu_1 \preceq \cdots \preceq \mu_k=\mu$, 
	and a sequence of rotations $\rho_1, \cdots, \rho_k$, so that $\rho_i = (\rho_i^-, \rho_i^+)=(\mu_{i-1} \setminus \mu_i, \mu_i \setminus \mu_{i-1})$ for each $1\leq i\leq k$. 
	Therefore, $\chi^{\mu_i} = \chi^{\mu_{i-1}} +(\chi^{\rho_i^+} - \chi^{\rho_i^-})$, which by recursion gives us $\chi^{\mu} = \chi^{\mu_0} + \sum_{i=1}^k (\chi^{\rho_i^+} - \chi^{\rho_i^-})$. 
	By Lemma~\ref{rotation}~(1), sets $\Pi(\mu)$ and $\{\rho_1, \cdots, \rho_k\}$ are equal with no repeated elements. This completes the proof.
\end{proof}

As remarked before, this result immediately implies our desired bound, by Proposition~\ref{thr:order}. 

\begin{corollary}
	Stable matching polytopes satisfy Conjecture~\ref{main-conjecture}.
\end{corollary}

We conclude the section with a remark on Theorem~\ref{marriage}. 
Even though its proof is relatively brief, to the best of the authors' knowledge this explicit connection was absent in the (extensive) literature of the problem,  
and it seems to simplify known results as well as shed new light on the structure of the stable matching polytope $S({\cal M})$. 
In particular, Eirinakis et al.~\cite{eirinakis2013polyhedral} have recently obtained for the first time the dimension, the number of facets, and a complete minimal linear description of $S({\cal M})$. 
Their analysis, based on the study of the rotation poset $\Pi$, as well as on ``reduced non-removable sets of non-stable pairs", is far from trivial. 
In contrast, our observation is theoretically simpler and immediately provides those results,  
as the facial structure of order polytopes is very well understood, and a simple minimal linear description of it is known; see~\cite{Stanley86}. 
Moreover, our result is also algorithmically significant, as it provides, from the rotation poset $\Pi$, a non-redundant system of equations and inequalities of $S({\cal M})$; 
and $\Pi$ can be efficiently constructed from the preference lists, in time $O(n^2)$ \cite[Lemma 3.3.2]{gusfield1989stable}.

\section{$2$-Level Matroid Base Polytopes}\label{sec:matroids}
For basic definitions and facts about matroids not appearing in the current section we refer to~\cite{oxley2006matroid}.

\subsection{2-level matroid polytopes and Conjecture~\ref{main-conjecture}}
In this section we give the relevant background on matroids and we prove that Conjecture~\ref{main-conjecture} holds for 2-level base polytope of matroids.

We identify a matroid $M$ by the tuple $(E,\B)$, where $E=E(M)$ is its ground set, and $\B=\B(M)$ is its base set. Whenever it is convenient, we describe a matroid in terms of its independent sets or its rank function. Given $M=(E,\B)$ and a set $F\subseteq E$, the \textit{restriction} $M|F$ is the matroid with ground set $F$ and independent sets $\I(M|F)=\{I\in\I(M) : I\subseteq F\}$; and the \textit{contraction} $M/F$ is the matroid with ground set $M\setminus F$ and rank function $r_{M/F}(A)=r_M(A\cup F)-r_M(F)$. For an element $e\in E$, the \emph{removal of $e$} is $M-e=M|(E-e)$. 
An element $p\in E$ is called a \emph{loop} (respectively \emph{coloop}) of $M$ if it appears in none (all) of the bases of $M$. 

Consider matroids $M_1=(E_1,\B_1)$ and $M_2=(E_2,\B_2)$, with non-empty base sets. If $E_1 \cap E_2 =\emptyset$, we can define the \emph{direct sum} $M_1 \oplus M_2$ as the matroid with ground set $E_1 \cup E_2$ and base set $\B_1 \times \B_2$. If, instead, $E_1 \cap E_2 = \{p\}$, where $p$ is neither a loop nor a coloop in $M_1$ or $M_2$, we let the \emph{2-sum} $M_1 \oplus_2 M_2$ be the matroid with ground set $E_1 \cup E_2 - p$, and base set $\{B_1\cup B_2 - p \ : \ B_i\in \B_i \text{ for } i=1,2 \text{ and } p\in B_1 \triangle B_2\}$.  A matroid is \emph{connected} (2-connected for some authors) if it cannot be written as the direct sum of two matroids, each with fewer elements; and a connected matroid $M$ is \textit{3-connected} if it cannot be written as a 2-sum of two matroids, both with strictly fewer elements than $M$.

%

%
%

The proofs of the following facts can be found e.g. in~\cite{oxley2006matroid}.

\begin{proposition}\label{obs:2-sum-facts}
	Let $M=M_1\s M_2$, with $E(M_1)\cap E(M_2)=\{p\}$.
	\begin{enumerate}
		\item\label{obs:2sum-part-1} $M_1 \oplus_2 M_2$ is connected if and only if so are $M_1$ and $M_2$.
		\item\label{obs:2sum-part-2} $\B(M_1\s M_2)=\B(M_1-p)\times \B(M_2/p) \uplus \B(M_1/p)\times \B(M_2-p)$. 
		\item\label{obs:2sum-part-2a} $|\B(M_i)|=|\B(M_i-p)|+|\B(M_i/p)|$, for $i=1,2$.
		\item\label{obs:2sum-part-3} If $M_2=M_2'\oplus M_2''$, where $E(M_1)\cap E(M_2')=\{p\}$ and $(E(M_1)\cup E(M_2'))\cap E(M_2'')=\emptyset$, then $M_1 \s M_2=(M_1\s M_2')\oplus M_2''.$
	\end{enumerate}
\end{proposition}

The \emph{base polytope} $B(M)\subseteq \R^E$ of a matroid $M=(E,\B)$ is given by the convex hull of the characteristic vectors of its bases. For a matroid $M$, the following is known to be a description of $B(M)$.
\begin{equation}\label{eq:base}
	\begin{array}{llllrllll} B(M) =  \{ x \in [0,1]^E : &  x(F) & \leq & r(F) & \hbox{for $F\subseteq E$;} &  x(E) & = & r(E)& \}.
	\end{array}
\end{equation}

A matroid $M(E,\B)$ is \emph{uniform} if $\B=\binom{E}{k}$, where $k$ is the rank of $M$. We denote the uniform matroid with $n$ elements and rank $k$ by $U_{n,k}$.
It is easy to check that the base polytope of a uniform matroid is a hypersimplex, i.e. $B(U_{n,k})=\{x\in \R^n: 0\leq x\leq 1, \sum_1^n {x_i}=k\}$.
Notice that, if $M_1$ and $M_2$ are uniform matroids with $|E(M_1)\cap E(M_2)|=1$, then $M_1\oplus_2 M_2$ is unique up to isomorphism, for any possible common element. 

\begin{figure}[h]
	\begin{center}
		\includegraphics[scale=0.75]{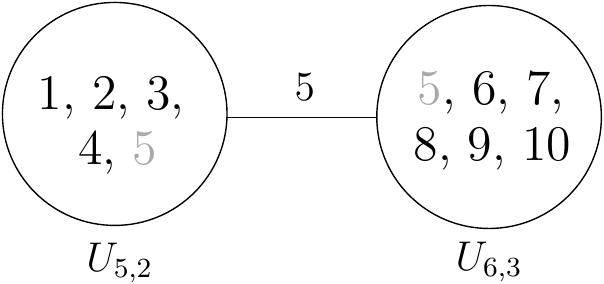}
	\end{center}
	\caption{A representation of $M=U_{5,2}\oplus_2 U_{6,3}$. $M$ has ground set $\{1,2,3,4,6,7,8,9,10\}$ and rank $4$, and two of its bases are $\{1,2,6,7\}$ and $\{1,6,7,8\}$. $B(M)$ is 2-level (see Theorem~\ref{thr:matroid-2-level}).} 
	\label{2sumuniform}
\end{figure}

Let $\M$ be the class  of matroids whose base polytope is 2-level. $\M$ has been characterized in~\cite{Grande14}:

\begin{theorem}\label{thr:matroid-2-level}
	The base polytope of a matroid $M$ is $2$-level if and only if $M$ can be obtained from uniform matroids through a sequence of direct sums and 2-sums.
\end{theorem}

The following lemma implies that we can, when looking at matroids in ${\cal M}$, decouple the operations of $2$-sum and direct sum. 

\begin{lemma}\label{lem:2-sums-then-product}
	Let $M$ be a matroid obtained by applying a sequence of direct sums and $2$-sums from the matroids $M_1,\dots,M_k$. Then $M=M'_1 \oplus M_2' \oplus ... \oplus M_t'$, where each of the $M_i'$ is obtained by repeated $2$-sums from some of the matroids $M_1,\dots,M_k$.
\end{lemma}

\begin{proof}
	Immediately from repeated applications of Proposition~\ref{obs:2-sum-facts}, part~\ref{obs:2sum-part-3}. \end{proof}

\begin{proposition}\label{obs:uniform(mod)}
	Let $M\in\M$ be connected and non-uniform, with $M=U_1\oplus_2\dots U_t$, where $U_i$ are uniform matroids and $t>1$. Then we can assume without loss of generality that every $U_i$ has at least 3 elements.
\end{proposition}
\begin{proof} No matroid in a 2-sum can have ground set of size one, since the 2-sum is defined when the common element is not a loop or a coloop of either summand. For the same reason, we can exclude the matroids $U_{2,0}, U_{2,2}$. The only remaining uniform matroid on two elements is $U_{2,1}$. However, it is easy to see that for any matroid $M$, $M\oplus_2 U_{2,1}$ is isomorphic to $M$: if the ground set of $U_{2,1}$ is $\{p,e\}$, with $p$ being the element common to $M$, the 2-sum has the only effect of replacing $p$ by $e$ in $M$.
\end{proof}

The following fact can be easily derived from \cite[Lemma 3.4]{Grande14}, but for completeness we give a self-contained proof in Appendix~\ref{app:2sumpolydescription}. 

\begin{lemma}\label{obs:2sumpolydescription}
	Let $M_1(E_1,\B_1), M_2(E_2,\B_2)$ be matroids with $E_1\cap E_2=\{p\}$ and let $M=M_1\oplus_2 M_2$. Then $B(M)$ is linearly isomorphic to $B(M_1)\times B(M_2)\cap\{x\in\R^{E_1\uplus E_2}: x_{p_1}+x_{p_2}=1\}$, where $E_1\uplus E_2 = E_1 \cup E_2 \cup \{ p_1, p_2\} -p$ is the disjoint union of $E_1$ and $E_2$, with $p_1$ and $p_2$ corresponding to $p \in E_1$ and $p\in E_2$ respectively.
\end{lemma}

\begin{proposition}\label{lem:facet-defining}
	Let $M\in \cal{M}$ be such that $M=M_1\s U$ where $U=U_{n,k}$ is a 3-connected uniform matroid with $n\geq 3$. Then $f_{d-1}(B(M))\leq f_{d_1 -1}(B(M_1))+2(n-1)$; and if $n=3$ then $f_{d-1}(B(M))\leq f_{d_1 -1}(B(M_1))+2$.
\end{proposition}
\begin{proof}
	Using Lemma~\ref{obs:2sumpolydescription}, we obtain that $B(M)$ is linearly isomorphic to $Q=B(M_1)\times B(U) \cap \{x\in\R^{E_1\uplus E_2}: x_{p_1}+x_{p_2}=1\}$, where $E_1,E_2, p, p_1,p_2$ are defined as before. From this it follows that $f_{d-1}(B(M))\leq f_{d_1 -1}(B(M_1))+f_{d_2 -1}(B(U))$, where $d_2$ is the dimension of $B(U)$. Moreover, as already remarked $B(U)=\{x\in \R^{d_2}: 0\leq x \leq 1, \sum_i x_i=k\}$ hence $f_{d_2 -1}(B(U))\leq 2n$ and $f_{d-1}(B(M))\leq f_{d_1 -1}(B(M_1))+2n$. To slightly sharpen the bound, we claim that the inequalities $0\leq x_{p_2} \leq 1$ present in the description of $Q$ are redundant, which proves the first part of the thesis. Indeed, they are immediately implied by the inequalities $0\leq x_{p_1}\leq 1$ (which must be implied by the description of $B(M_1)$) together with the equation $x_{p_1}+x_{p_2}=1$.
	
	We now consider the case $n=3$. It is immediate to check that there are two cases, $U=U_{3,1}$ and $U=U_{3,2}$, but for both $B(U)$ is isomorphic to a triangle in the plane, and hence $f_{d_2 -1}(B(U))=3$, with one inequality for each variable: for instance, a description of $B(U_{3,1})$ is $\{x\in \R^3: x\geq 0, x_1+x_2+x_3=1\}$. Arguing as before, we obtain that in the resulting description of $Q$ the inequality relative to $x_{p_2}$ is redundant, thus getting the desired bound.
\end{proof}

\begin{lemma}\label{thr:conj-matroid}
	2-level matroid base polytopes satisfy Conjecture~\ref{main-conjecture}.
\end{lemma}
\begin{proof} 
	We will use the fact that, for any $n\geq 3$ and any $k\in\{0,\dots,n\}$, $\binom{n}{k}\leq \frac{3}{4}2^{n-1}$. This can be easily proved by induction. We prove the conjecture on the polytope $B(M)$, for each matroid $M=(E,\B)\in \M$, and we prove it by induction on the number of elements $n=|E|$. The base cases $n\leq 3$ can be easily verified.
	
	If $M$ is not connected, then $M=M_1\oplus M_2$ for two matroids $M_1,M_2\in\M$, each with fewer elements than $M$, so by induction hypothesis the conjecture holds for them. The base polytope $B(M)$ is simply the Cartesian product of $B(M_1)$ and $B(M_2)$, so by Lemma~\ref{lem:cartesian-products-are-good} the conjecture also holds for $B(M)$, and is tight only if $B(M)$ is a cube.
	
	Assume from now on that $M$ is connected. In~\cite{Grande14}, it is proven that the smallest affine subspace containing the base polytope of a connected matroid on $n$ elements is of dimension $d=n-1$. If $M$ is  uniform, $M=U_{n,k}$, the number of vertices in $B(M)$ is $f_0=|\B|=\binom{n}{k}\leq \frac{3}{4} 2^{n-1}$, where we assumed $n\geq 3$. And in view of Proposition~\ref{obs:uniform(mod)}, the constraints of the form $0\leq x\leq 1$ are sufficient to define $B(M)$, hence the number of facets is $f_{d-1}\leq 2n$. Therefore, $f_0 f_{d-1}\leq \frac{3}{4} n2^n \leq (n-1)2^n =d2^{d+1}$, where the last inequality is loose for $n\geq 5$. The only examples with $n\leq 4$ for which the conjecture is tight correspond to cubes, and the 3-dimensional cross-polytope coming from $U_{4,2}$.
	
	Finally, assume that $M$ is connected but is not uniform, so it is not 3-connected. Then $M=M_1\s M_2$, with matroids $M_1,M_2\in\M$ each with fewer elements than $M$, so by induction hypothesis the conjecture holds for both of them. Let $E(M_1)\cap E(M_2)=\{p\}$. Both $M_1$ and $M_2$ are connected, by Proposition~\ref{obs:2-sum-facts}. We can assume without loss of generality that $E(M_1)=n_1\geq n_2=E(M_2)$, and that $M_2$ is uniform, $M_2=U_{n_2,k_2}$, with $n_2\geq 3$ (by Proposition~\ref{obs:uniform(mod)}). We consider two cases for the value of $n_2$.
	
	\smallskip
	Case $n_2\geq 4$: first notice that the family $\M$ is closed under removing or contracting an element. This is because if $e\in M\in\M$, the base polytopes $B(M-e)$ and $B(M/e)$ are affinely isomorphic to the faces of $B(M)$ that intersect the hyperplanes $x_e=0$ and $x_e=1$, respectively, and by Lemma~\ref{lem:2-level-2^d} these faces are also 2-level. Hence, we know from Proposition~\ref{obs:2-sum-facts} that 
	\begin{align*}
		f_0 &=|\B(M)|=|\B_{M_1 -p}|\cdot |\B_{U_{n_2,k_2}/p}|+|\B_{M_1 /p}|\cdot |\B_{U_{n_2,k_2}-p}| \\
		&=\binom{n_2-1}{k_2-1}|\B_{M_1 -p}|+\binom{n_2-1}{k_2}|\B_{M_1 /p}| \\
		&\leq \frac{3}{4}2^{n_2 -2}\left(|\B_{M_1 -p}|+|\B_{M_1 /p}|\right) = \frac{3}{4} 2^{d_2-1}|\B(M_1)|.
	\end{align*}
	From Proposition~\ref{lem:facet-defining}, the number of facets in $B(M)$ is $$f_{d-1}(B(M))\leq f_{d_1 -1}(B(M_1))+2(n_2-1)=f_{d_1 -1}(B(M_1))+2d_2.$$ We use the induction hypothesis in $M_1$, and the trivial bound $|\B(M_1)|\leq 2^{d_1}$ to obtain:
	\begin{align*}
		f_0f_{d-1}(B(M))&<\frac{3}{4}2^{d_2-1}|\B(M_1)|\left(f_{d_1 -1}(B(M_1))+2d_2\right)\\
		&\leq \frac{3}{4}2^{d_2-1}\left(d_12^{d_1+1} +2^{d_1}(2d_2)\right)\\
		&=\frac{3}{4}(d_1+d_2)2^{d_1+d_2}<(d_1+d_2-1)2^{d_1+d_2}=d2^{d+1}. 
	\end{align*}
	Where in the last inequality we used the fact that $n_1\geq n_2\geq 4$, so $d_1\geq d_2\geq 3$.
	
	\smallskip
	Case $n_2=3$: We can prove in a similar manner as before that
	$$f_0=|\B(M)|<\binom{2}{1} \left(|\B(M_1 - p)|+|\B(M_1 / p)|\right)=2|\B(M_1)|.$$
	And from Proposition~\ref{lem:facet-defining}, $f_{d-1}(B(M))\leq f_{d_1 -1}(B(M_1))+2$. Thus,
	$$f_0f_{d-1}(B(M))<2|\B(M_1)|\left(f_{d_1 -1}(B(M_1))+2\right)\leq 2\left(d_1 2^{d_1+1}+ 2^{d_1}\cdot 2\right)=d2^{d+1}.$$
	We conclude by remarking that, since the inequalities above hold strictly, the only 2-level base polytopes satisfying the bound of Conjecture~\ref{main-conjecture} are cubes and cross-polytopes.
\end{proof}

As the forest matroid of a graph $G$ is in $\M$ if and only if $G$ is series-parallel~\cite{Grande14}, we deduce the following.

\begin{corollary}\label{cor:spanningtreeSP}
	Conjecture~\ref{main-conjecture} is true for the spanning tree polytope of series-parallel graphs.
\end{corollary}

\subsection{Linear Description of 2-Level Matroid Base Polytopes}\label{sec:basepolytope}
With the help of Proposition~\ref{lem:facet-defining}, one can easily prove by induction that for any $M\in {\cal M}$ the number of facets of $B(M)$ is linear in the size of the ground set. However, the description of $B(M)$ given in \eqref{eq:base} has exponentially many inequalities. Finding compact description for the base and the independent set polytopes of matroids has been the object of many studies, especially in terms of extended formulations: see~\cite{rothvoss2013some} for a negative result, and~\cite{conforti2015subgraph},~\cite{iwata2016extended},~\cite{kaibel2016extended} for formulations for special classes of matroids. In particular in~\cite{kaibel2016extended} a polynomial size (extended) formulation is given for the class of regular matroids which relies on structural results of Seymour~\cite{seymour1980decomposition} and can be obtained in polynomial time given an independence oracle for the matroid. These results can be seen as generalizations of the formulations given for the spanning tree polytope by Martin~\cite{martin1991using}.
In this section we give an explicit description of 2-level base matroids with linearly many inequalities. The rank inequalities needed in our description have a natural interpretations in terms of the combinatorial structure of the matroid, in a similar fashion as in~\cite{kaibel2016extended}. Our description can also be obtained in polynomial time. 

Since the base polytope of the direct sum of matroids is the Cartesian product of the base polytopes, to obtain a linear description of $B(M)$ for $M\in \cal M$, we can focus on base polytopes of connected matroids.
Any connected matroid can be seen as a sequence of $2$-sums, which can be represented via a tree (see Figure~\ref{fig}): the following is a version of \cite[Proposition 8.3.5]{oxley2006matroid} tailored to our needs. 
For completeness, a proof is given in Appendix~\ref{app:oxley2006matroid-tree}. 

\begin{theorem}
	\label{thm:oxley2006matroid-tree}
	Let $M$ be a connected matroid. Then there are 3-connected matroids $M_1,\dots M_t$, and a $t$-vertex tree $T=T(M)$ with edges labeled $e_1,\dots, e_{t-1}$ and vertices labeled $M_1,\dots, M_t$, such that
	\begin{enumerate}
		\item $E(M)\cap\{e_1,\dots,e_{t-1}\}=\emptyset$, and  $E(M_1)\cup E(M_2) \cup \cdots \cup E(M_t)=E(M)\cup \{e_1,\dots,e_{t-1}\}$;
		\item if the edge $e_i$ joins the vertices $M_{j_1}$ and $M_{j_2}$, then $E( M_{j_1})\cap E(M_{j_2})=\{e_i\}$;
		\item if no edge  joins the vertices $M_{j_1}$ and $M_{j_2}$, then $E( M_{j_1})\cap E(M_{j_2})=\emptyset$.
	\end{enumerate}
	Moreover, $M$ is the matroid that labels the single vertex of the tree $T/e_1,\dots,e_{t-1}$ at the conclusion of the following process: contract the edges $e_1,\dots,e_{t-1}$ of $T$ one by one in order; when $e_i$ is contracted, its ends are identified and the vertex formed by this identification is labeled by the 2-sum of the matroids that previously labeled the ends of $e_i$.
\end{theorem}

\begin{eexample}\label{exam2}
	Consider the matroid $M$ whose associated tree structure is given in Figure~\ref{fig}. The ground set of $M$ is $\{1,2,3,4,8,9,10,11,12,13,14,15\}$ and its rank, which can be computed as the sum of the ranks of the nodes minus the number of edges, is 4. $\{1,2,11,13\}$ is a basis.
\end{eexample}

\begin{figure}[h]
	\begin{center}
		\includegraphics[scale=0.7]{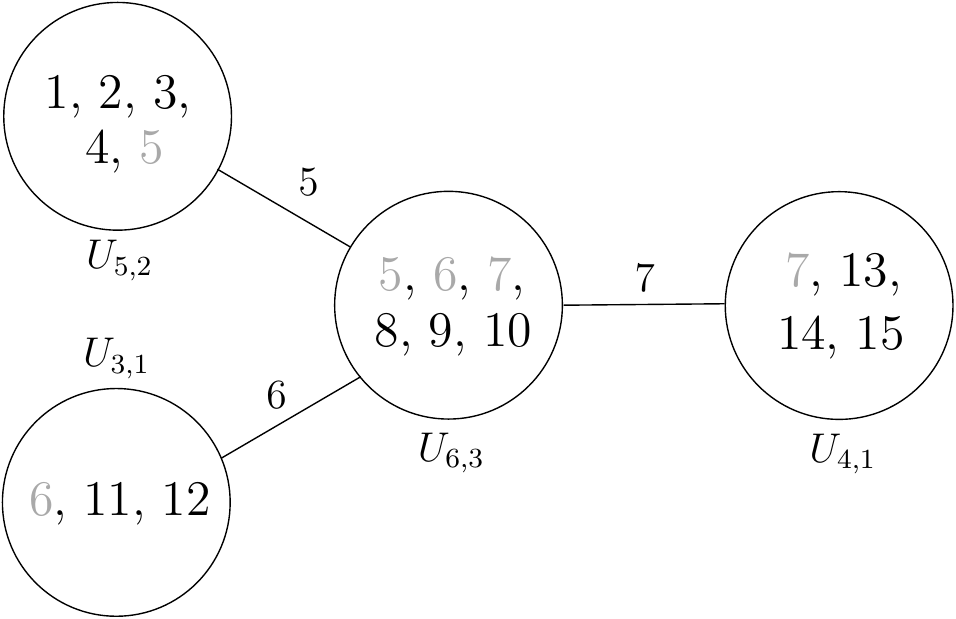}
	\end{center}
	\caption{The matroid from Example~\ref{exam2}.}
	\label{fig}
\end{figure}

For a connected matroid $M(E,{\cal B}) \in\M$, Theorem~\ref{thm:oxley2006matroid-tree} reveals a tree structure $T(M)$, where every node represents a 3-connected uniform matroid, and every edge represents a 2-sum operation. We now give a simple description of the associated base polytope. Let $a$ be an edge of $T(M)$. The removal of $a$ breaks $T$ into two connected components ${C}^1_a$ and ${C}^2_a$. Let $E^1_a$ (resp. $E^2_a$) be the set of elements from $E$ that belong to uniform matroids from ${C}^1_a$ (resp. ${C}^2_a$).  The following theorem shows that the inequalities needed to describe $B(M)$ are the ``trivial" inequalities $0\leq x\leq 1$, plus $x(F)\leq r(F)$, where $F=E^1_a$ or $E^2_a$ for some edge $a$ of $T(M)$. If $M$ is 2-sum of uniform matroids $U_1,\dots, U_t$, then clearly $T$ will have $t-1$ edges. From Proposition~\ref{obs:uniform(mod)}, we know that  $E(U_i)\geq 3$ for any $i$. Hence, if $|E|=n$, we have $$n=\sum_{i=1}^t |E(U_i)|-2(t-1)\geq 3t-2(t-1)=t+2,$$ hence $t\leq n-2$. Thus, the total number of inequalities needed is linear in the number of elements.


\begin{theorem}\label{thr:matroid}
	
	Let $M=(E,{\cal B}) \in {\cal M}$ be a connected matroid obtained as $2$-sum of uniform matroids $U_1=U_{n_1,k_1}, \dots, U_t=U_{n_t,k_t}$. Let $T(N,A)$ be the tree structure of $M$ according to Theorem~\ref{thm:oxley2006matroid-tree}. For each $a \in A$,  let ${C}^1_a$, ${C}^2_a$, $E_a^1, E_a^2$ be defined as above. 
	Then 
	$$\begin{array}{rllr} B(M) =  \{ x \in \R^E : x & \geq & 0 \\ [1.02ex]
	 x & \leq & 1 \\[1.02ex] 
	 x(F) & \leq & \rk(F) & \hbox{ for } F=E^i_a \hbox{ for some } i\in \{1,2\} \hbox{ and } a \in A, \\ [1.02ex]
	x(E) & = & \rk(E) & \}.\end{array}$$
	
	\noindent Moreover, if $F=E_a^i$ for some $i\in\{1,2\}$, $a \in A$, then $\rk(F) = 1 - |C_a^{i}| + \sum_{ j : U_j \in {\cal C}_a^i} k_j$.
\end{theorem}
\begin{proof}
	Let $M=U_1\oplus_2\dots \oplus_2 U_t$ and $T(N,A)$ be as in the hypotheses. We proceed by induction on $t$. If $t=1$, there is nothing to prove as $M$ is uniform and $A=\emptyset$. 
	Let $t>1$, and assume without loss of generality that the node $U_t$ of $T$ is a leaf, or in other words that $E(U_t)\cap \cup_{i=1}^{t-1}E(U_i)$ consists of only one element, which we denote by $p$. Then we can write $M$ as a 2-sum of $M_1=U_1\oplus_2\dots\oplus_2 U_{t-1}$, with ground set $E_1$, and $M_2=U_t$, with ground set $E_2$, with $E_1\cup E_2-p=E$. 
	From Proposition~\ref{obs:2-sum-facts}, part 1, we have that $M_1\in \cal M$ is connected, hence it satisfies the induction hypothesis with tree structure $T_1$, the subtree of $T$ induced by nodes $U_1,\dots, U_{t-1}$. Let $A_1$ be the edge set of $T_1$. For any edge $a\in A_1$, $T_1-a$ has connected components $\tilde{C}_a^1=C_a^1, \tilde{C}_a^2= C_a^2-U_t$, and $\tilde E_a^1,\tilde E_a^2$ are defined accordingly. Using Lemma~\ref{obs:2sumpolydescription}, we have that $B(M)$ is isomorphic to
	$$
	Q=B(M_1)\times B(U)\cap \{y\in\R^{E_1\uplus E_2}: y_{p_1}+y_{p_2}=1\},
	$$
	where $E_1\uplus E_2=E_1\cup E_2\cup\{p_1,p_2\}-p$ as in Lemma~\ref{obs:2sumpolydescription}. We use $y$ for variables in $\R^{E_1\uplus E_2}$ and $x$ for variables in $\R^E$ to avoid confusion. 
	From the induction hypothesis, $Q$ and  can be described as follows:
	$$\begin{array}{lrlllr} Q =  \{ y \in \R^{E_1\uplus E_2} : && 0\leq y_e\leq 1 &  \hbox{ for } e\in E_1 \\ [1.02ex]
	&& y(F)  \leq \rk(F) & \hbox{ for } F=\tilde E^i_a, \ i\in\{1,2\}, \ a \in A_1, \\ [1.02ex]
	&& y(E_1)  = \rk(E_1)&  \\[1.02ex]
	&& 0\leq y_e\leq 1 &  \hbox{ for } e\in E_2-p_2 \\ [1.02ex]
	&&y(E_2)= \rk(E_2)&\\ [1.02ex]
	&& y_{p_1}+y_{p_2}=1 &&&\},\end{array}$$
	where we excluded $0\leq y_{p_2}\leq 1$ as it is implied by the system (see the proof of Proposition~\ref{lem:facet-defining}).  
	Let $\varphi$ be the projection from $\R^{E_1\uplus E_2}$ to $\R^{E}$, as in the proof of Lemma~\ref{obs:2sumpolydescription}. 
	We have that $\varphi$ is a bijection between $Q$ and $B(M)$, and it also induces a bijection between the faces of $Q$ and those of $B(M)$. 
	To complete the proof, we just need to show that, for any face $\calF$ of $Q$ corresponding to an inequality given above, the face $\varphi(\calF)$ of $B(M)$ is described by one of the inequalities given in the thesis. 
	
	First, for any $e\in E=E_1\uplus E_2\setminus\{p_1,p_2\}$, let $\calF_{e,0}=\{y\in\R^{E_1\uplus E_2}: y_e=0\}\cap Q$. 
	It is immediate to see that $\varphi(\calF_{e,0})=\{x\in \R^E: x_e=0\}\cap B(M)$. Similarly for the faces $\calF_{e,1}$ induced by $y_e=1, e\in E$. 
	Consider now $\calF_{p_1,0}$; we claim that $$\varphi(\calF_{p_1,0})=\{x\in \R^E: x(E_1-p)=\rk(E_1-p)\}\cap B(M).$$ 
	Indeed, $x\in \varphi(\calF_{p_1,0})$ if and only if $x=\varphi(y)$ with $y\in Q$ and $y_{p_1}=0$. 
	The last equation is equivalent to $y(E_1)=y(E_1-p_1)=x(E_1-p)$, i.e., $x(E_1-p)=\rk(E_1)=\rk(E_1-p)$, which holds since $M_1$ is connected. 
	In the same way we can see that $$\varphi(\calF_{p_1,1})=\{x\in \R^E: x(E_2-p)=\rk(E_2-p)\}\cap B(M).$$ 
	Let $U_\ell$ be the unique neighbor of $U_t$ in $T$. 
	The two inequalities just described correspond to $x(F) \leq \rk(F)$ for $F=E^i_{\bar{a}}$, where $\bar{a}$ is the edge between $U_\ell$ and $U_t$ in $T$, and $i=1,2$. 
	We now consider the inequalities corresponding to the other edges of $T$ (which are edges of $T_1$ as well).
	For any such edge $a$, let $\calF_{a,i}=\{y\in\R^{E_1\uplus E_2}: y(\tilde E^i_a)=\rk(\tilde E_a^i\}\cap Q$, for $i=1,2$. We claim that 
	$$ \varphi(\calF_{a,1})=\{x\in \R^E: x(E^1_a)=\rk(E_a^1)\}\cap B(M),$$
	$$\varphi(\calF_{a,2})=\{x\in \R^E: x(E^2_a)=\rk(E_a^2)\}\cap B(M).$$
	If, among $\tilde C_a^1, \tilde C_a^2$, the latter is such that $U_\ell\in \tilde C_a^2$, and $E_a^1,E_a^2,\tilde E_a^1,\tilde E_a^2$ are defined accordingly, then we have $E^1_a=\tilde E^1_a$ and the first equality is immediate. For the second equality, we argue similarly as before, exploiting the fact that any connected subtree of $T$ gives a connected matroid that is 2-sum of its nodes. Let $M_a$, $\tilde M_a$, be obtained as 2-sums of the matroids in $C_a^2, \tilde C_a^2$ respectively. Then one has $M_a=\tilde M_a\oplus_2 M_2$, (notice that $E^2_a=\tilde E^2_a\cup E_2-p$), which implies $\rk(E_a^2)=\rk(\tilde E_a^2)+\rk(E_2)-1$. We have $x\in \varphi(\calF_{a,2})$ if and only if $x=\varphi(y)$ with $y\in Q$ and $y(\tilde E^2_a)=\rk(\tilde E^2_a)$. Now, if $y_{p_1}=0$, then $y_{p_2}=1$, and one has $x(\tilde E_a^2-p)=y(\tilde E^2_a)=\rk(\tilde E^2_a)$ and $x(E_2-p)=y(E_2)-1=\rk(E_2)-1$, which implies $x(E_a^2)=x(\tilde E_a^2-p)+x(E_2-p)=\rk(E_a^2)$. If $y_{p_1}=1$ and $y_{p_2}=0$, one has $x(\tilde E_a^2-p)=y(\tilde E^2_a)-1=\rk(\tilde E^2_a)-1$, and $x(E_2-p)=y(E_2)=\rk(E_2)$, which again implies $x(E_a^2)=\rk(E_a^2)$. The reverse implication, that $x(E_a^2)=\rk(E_a^2)$ implies $y(\tilde E^2_a)=\rk(\tilde E^2_a)$, can be shown in the same way, and this completes the proof.
\end{proof}

We conclude by remarking that, for any matroid $M$, the corresponding tree structure given in Theorem~\ref{thm:oxley2006matroid-tree} can be obtained in polynomial time, given an independence oracle for $M$, for instance using the shifting algorithm given in~\cite{bixby1996matroid}. This means that, given an independence oracle for $M\in \cal M$, one can efficiently write down the description of $B(M)$ given by Theorem~\ref{thr:matroid}: first, one obtains the tree structure and the corresponding uniform matroids, and then the rank inequalities corresponding to the edges of the tree. The latter part just takes linear time in the number of elements of $M$.


\section{Cut Polytope and Matroid Cycle Polytope}\label{sec:cycle-matroid} 

Given a graph $G$ with edge set $E$, its \emph{cut polytope} $CUT(G)\subseteq \mathbb{R}^E$ is the convex hull of the characteristic vectors of the cuts of $G$. For general graphs, a linear description of $CUT(G)$ is not known. However, for graphs without $K_5$ as a minor, $CUT(G)$ is described by:
\begin{equation}\label{eq:cut}
	CUT(G) =  \{ x \in [0,1]^E : x(F) - x(C\setminus F)\leq |F|-1 \; \; \forall F\in {\cal F}\}, 
\end{equation}
where ${\cal F}=\{ F\subset V(G): F\subset C, \; C \text{ induced cycle of } G, \; |F| \text{ odd} \}$.

For a matroid $M=(E,{\cal B})$, a set $C\subseteq E$ is a \emph{cycle} if $C=\emptyset$ or $C$ is a disjoint union of circuits. The \emph{cycle polytope} $C(M)$ of $M$ is the convex hull of the characteristic vectors of its cycles~\cite{barahona1986cycle}. Cycle polytopes can be seen as a generalization of cut polytopes. Indeed, it can be shown that if $M$ is cographic, i.e. it is the dual of the forest matroid of some graph $G$, then the cycles of $M$ correspond to the cuts of $G$, hence $C(M)=CUT(G)$. 
The \emph{cycle polytope} $C(M)$ is given by the convex hull of the characteristic vectors of its cycles, and it is a generalization of the cut polytope $CUT(G)$ for a graph $G$~\cite{barahona1986cycle}. 

A matroid is called binary if it can be represented over the finite field $GF_2$. Given a matroid $M$, we denote by $M^*$ its dual matroid. $M$ is binary if and only if $M^*$ is binary. 
An element $e$ of a matroid is a chord of a circuit $C$ if $C$ is the symmetric difference of two circuits whose intersection is $e$. A chordless circuit is a circuit with no chords and the same definition can be applied to cocircuits, that are circuits in the dual matroid. $F_7^*$ denotes the dual of the Fano matroid; $R_{10}$ is a binary matroid associated with the $5 \times 10$ matrix whose columns are the 10 $0/1$ vectors with 3 ones and 2 zeros;  $M_{K_5}^*$ is the dual of the forest matroid of $K_5$.  

In this section we prove Conjecture~\ref{main-conjecture} for the cycle polytope $C(M)$ of the binary matroids $M$ that have no minor isomorphic to $F_7^*$, $R_{10}$, $M_{K_5}^*$ and are $2$-level. When those minors are forbidden, a complete linear description of the associated polytope is known (see~\cite{barahona1986cycle}). This class includes all cut polytopes that are $2$-level, and has been characterized in~\cite{gouveia2012new}:

\begin{theorem} \label{thr:cycle-polytope-2-level}
	Let $M$ be a binary matroid with no minor isomorphic to $F_7^*$, $R_{10}$, $M_{K_5}^*$. Then $C(M)$ is 2-level if and only if $M$ has no chordless cocircuit of length at least 5.
\end{theorem}

\begin{corollary}
	The polytope $CUT(G)$ is 2-level if and only if $G$ has no minor isomorphic to $K_5$ and no induced cycle of length at least 5.
\end{corollary}

Recall that the cycle space of graph $G$ is the set of its Eulerian subgraphs  (subgraphs where all vertices have even degree), 
and it is known (see for instance~\cite{gross2005graph}) to have a vector space structure over the field $\mathbb{Z}_2$. 
This statement and one of its proofs easily generalizes to the cycle space (the set of all cycles) of binary matroids. 
We provide a proof in Appendix~\ref{app:cycle-space-binary} for completeness.
\begin{lemma}\label{lem:cycle-space-binary}
	Let $M$ be a binary matroid with $d$ elements and rank $r$. Then the cycles of $M$ form a vector space $\cal C$ over $\mathbb{Z}_2$ with the operation of symmetric difference as sum. Moreover, $\cal C$ has dimension $d-r$.
\end{lemma}

\begin{corollary} \label{cor:cycles-bound}
	Let $M$ be a binary matroid with $d$ elements and rank $r$. Then $M$ has exactly $2^{d-r}$ cycles.
\end{corollary}

The only missing ingredient is a description of the facets of the cycle polytope for the class of our interest, which extends the description of the cut polytope given in~\ref{eq:cut}.

\begin{theorem}\cite{barahona1986cycle} \label{thr:barahona1986cycle}
	Let $M$ be a binary matroid, and let ${\overline{ \cal C}}$ be its family of chordless cocircuits. Then $M$ has no minor isomorphic to $F_7^*$, $R_{10}$, $M_{K_5}^*$ if and only if
	$$C(M) =  \{ x \in [0,1]^E : x(F) -x(C\setminus F) \leq |F|-1 \mbox{ for } C \in {\overline{ \cal C}}, F\subseteq C, |F| \mbox{ odd}\}.$$
\end{theorem}

\begin{lemma}\label{thr:cycle-polytope-conjecture}
	Let $M$ be a binary matroid with no minor isomorphic to $F_7^*$, $R_{10}$, $M_{K_5}^*$ and such that $C(M)$ is $2$-level. Then $C(M)$ satisfies Conjecture~\ref{main-conjecture}.
\end{lemma}
\begin{proof}
	As remarked in~\cite{barahona1986cycle} and~\cite{gouveia2012new}, the following equations are valid for $C(M)$: a) $x_e=0$, for $e$ coloop of $M$; and b) $x_e-x_f=0$, for $\{e,f\}$ cocircuit of $M$.
	
	The first equation is due to the fact that a coloop cannot be contained in a cycle, and the second to the fact that circuits and cocircuits have even intersection in binary matroids. A consequence of this is that we can delete all coloops and contract $e$ for any cocircuit $\{e,f\}$ without changing the cycle polytope: for simplicity we will just assume that $M$ has no coloops and no cocircuit of length 2. In this case $C(M)$ has full dimension $d=|E|$.  Let $r$ be the rank of $M$. Corollary~\ref{cor:cycles-bound} implies that $C(M)$ has $2^{d-r}$ vertices. Let now $T$ be the number of cotriangles (i.e., cocircuits of length 3) in $M$, and $S$ the number of cocircuits of length 4 in $M$. Thanks to Theorem~\ref{thr:barahona1986cycle} and to the fact that $M$ has no chordless cocircuit of length at least 5, we have that $C(M)$ has at most $2d+4T+8S$ facets. Hence the bound we need to show is:
	$$
	2^{d-r}(2d+4T+8S)\leq d2^{d+1}, \; \text{ which is equivalent to } \; 2T+4S\leq d(2^{r}-1).
	$$
	Since the cocircuits of $M$ are circuits in the binary matroid $M^*$, whose rank is $d-r$, we can apply Corollary~\ref{cor:cycles-bound} to get $T+S\leq 2^r-1$, where the $-1$ comes from the fact that we do not count the empty set. Hence, if $d\geq 4$,
	$$
	2T+4S\leq 4(T+S)\leq d(2^{r}-1).
	$$
	
	The bound is loose for $d\geq 5$. The cases with $d\leq 4$ can be easily verified, the only tight examples being affinely isomorphic to cubes and cross-polytopes.
\end{proof}

\begin{corollary}
	2-level cut polytopes satisfy Conjecture~\ref{main-conjecture}.
\end{corollary}

\section{On possible generalizations of the conjecture}\label{sec:conclusion}

In this paper, we provided a thorough analysis of $2$-level polytopes coming from combinatorial settings. 
We hope that the reader shares with us the opinion that those polytope are relevant for the mathematical community, and the $2$-levelness property seem to be strong enough to leave hope for deep theorems on their structure. 
While we proved Conjecture~\ref{main-conjecture} for all $2$-level polytopes we could characterize, it remains open for the general case. Whether some techniques and ideas introduced in this paper can be extended to attack it also remains open. Here, we would like to discuss a different issue stemming from Conjecture~\ref{main-conjecture}: is $2$-levelness the ``right" assumption for proving $f_{d-1}(P)f_0(P)\leq d 2^{d+1}$, or is this bound valid for a much more general class of $0/1$ polytopes -- or, more broadly, of mathematical objects? We start the investigation of this question by providing some examples of ``well-behaved" $0/1$ polytopes that do not verify Conjecture~\ref{main-conjecture}. The first two can be seen as immediate generalizations of polytopes for which Conjecture~\ref{main-conjecture} holds, see Corollary~\ref{cor:spanningtreeSP}.

\subsection{Forest polytope of $K_{2,n}$}
Let $P$ be the forest polytope of $K_{2,n}$. Note that $P$ has dimension $d=2n$. Conjecture~\ref{main-conjecture} implies an upper bound of $n 2^{2(n+1)}=O(4+\varepsilon)^n$ for $f_{0}(P)f_{d-1}(P)$, for any $\varepsilon >0$. Each subgraph of $K_{2,n}$ that takes, for each node $v$ of degree $2$, at most one edge incident to $v$, is a forest. Those graphs are $3^n$. Moreover, each induced subgraph of $K_{2,n}$ that takes the nodes of degree $n$ plus at least $2$ other nodes is $2$-connected, hence it induces a (distinct) facet of $P$. Those are $2^{n} -(n+1)$. In total $f_0(P) f_{d-1}(P) = \Omega(6^n)$.

\subsection{Spanning tree polytope of the skeleton of the $4$-dimensional cube}
Let $G$ be the skeleton of the $4$-dimensional cube, and $P$ the associated spanning tree polytope. Through extensive computation\footnote{We computed the number of spanning trees of $G$ using the well known Kirchhoff's matrix tree theorem~\cite{buekenhout1998number}. The facets of the spanning tree polytope of a 2-connected graph $G$ are roughly as many as the 2-connected, induced subgraphs of $G$ whose contraction is 2-connected, and we compute them by exhaustive search. The Matlab code can be found at: \url{http://disopt.epfl.ch/files/content/sites/disopt/files/users/249959/flacets.zip}}, we verified that $f_{0}(P) f_{d-1}(P) \geq 1.603 \cdot 10^{11}$, while the upper bound from Conjecture~\ref{main-conjecture} is $\approx1.331  \cdot 10^{11}$.

\subsection{$3$-level min up/down polytopes}\label{ex:minupdown}
Fix $d\geq 3$. A 0/1 vector $x\in\{0,1\}^d$ is ``bad" if there are indices $0<i<j<d$ such that $x_i=x_{j+1}=1$ and $x_{i+1}=x_j=0$. In other words, when seen as a bit-string, $x$ is bad if it contains two or more separate blocks of 1's. Let $P\subset\mathbb{R}^d$ be the convex hull of all 0/1 vectors that are not bad: this is a min up/down polytope, as defined in~\cite{lee2004min}, with parameters $\ell_1=1$ and $\ell_2=d-1$\footnote{Recall that the min up/down polytope is 2-level precisely when its parameters $\ell_1$ and $\ell_2$ are equal, and in that case the polytope satisfies Conjecture~\ref{main-conjecture}, see Proposition~\ref{thr:up/down}.}.

Each non-zero vertex $x$ in $P$ contains exactly one block of 1's, thus it is uniquely described by two indices $0\leq i<j\leq d$, such that $x_k=1$ if $i<k\leq j$, and $x_k=0$ otherwise. Therefore (counting also the zero vector), $P$ contains $\binom{d+1}{2}+1$ vertices. On the other hand, from the facet characterization presented in~\cite{lee2004min} we know that 
$$P=\Big\{x\in \mathbb{R}^d_+: \ \sum_{j=1}^k (-1)^{j-1} x_{i_j} \leq 1, \text{ for } 1\leq i_1<\cdots<i_k\leq d \text{ s.t. $k$ is odd}\Big\},$$
where all inequalities above are facet-defining. Moreover, since the polytope is full-dimensional (it contains the $d$-dimensional standard simplex) and no inequality is a multiple of another, they all define distinct facets.  This means that there are $d$ facets coming from non-negativity constraints, and $2^{d-1}$ facets that are in one-to-one correspondence with odd subsets of the index set $[d]$. Hence, the total number of facets is $2^{d-1}+d$. It is easy to check that for $d\geq 3$ we have
$$f_0 (P) f_{d-1}(P) = \left[\binom{d+1}{2}+1\right]\cdot[2^{d-1}+d]>d2^{d+1},$$
thus the polytope does not satisfy Conjecture~\ref{main-conjecture}. Note that $P$ is a \emph{3-level} polytope: for each facet $F$ of $P$, there exist two translates of the affine hull of $F$ such that all the vertices of $P$ lie either in $F$ or in one of those two translates.

\smallskip

In the remaining sections, we move to extensions of Conjecture~\ref{main-conjecture} to other settings. In some of those cases we could prove that the conjecture does not hold. Others are interesting open questions. 

\subsection{Polytopes of minimum PSD rank} $2$-level polytopes are an example of polytopes with minimum PSD rank, i.e. such that they admit a semidefinite extension of size $1+dim(P)$, see~\cite{GRT13}. A necessary and sufficient condition characterizing those polytopes is given in~\cite{GRT13}, where the full list of combinatorial classes of polytopes with minimum PSD rank in dimension 2 and 3 is also given. All those are combinatorially equivalent to some $2$-level polytope of the same dimension, with the exception of the bypiramid over a triangle, which clearly verifies Conjecture~\ref{main-conjecture}.  In~\cite{GPRT17}, the full list of combinatorial classes of polytopes with minimum PSD rank in dimension 4 is given. By going through the list of their $f$-vectors in \cite[Table 1]{GPRT17}, one easily checks that they also verify Conjecture~\ref{main-conjecture}. We are not aware of studies on higher-dimensional polytopes with minimum PSD rank. We remark that in~\cite{Grande14} it is proved that matroid base polytopes have minimum PSD rank if and only if they are 2-level, hence Lemma~\ref{thr:conj-matroid} trivially implies that all matroid polytopes with minimum PSD rank satisfy the bound of Conjecture~\ref{main-conjecture}.

\subsection{Polytopes with structured linear relaxations}\label{ex:LR} 
We consider possible generalizations of the conjecture based on the ${\cal H}-$embedding of $2$-level polytopes. Those are defined in~\cite{bohn2015enumeration}, where it is shown that the family of $2$-level polytopes of dimension $d$ is affinely equivalent to the family of integral polytopes of the form
\begin{equation}\label{eq:2L-embed}
	P=\{ x \in \R^d : 0\leq x(I) \leq 1 \quad \hbox{ for all }I\subseteq {\cal I}\}\end{equation}
for some ${\cal I}\subseteq 2^{[d]}$. Hence, Conjecture~\ref{main-conjecture} holds for $2$-level polytopes if and only if it holds for integral polytopes of the form \eqref{eq:2L-embed}.
First, notice that the bound of the conjecture does not hold for general (i.e., non-integral) polytopes of the form \eqref{eq:2L-embed}, as for instance fractional stable set polytopes are of such form. In particular, consider the fractional stable set polytope of the complete graph on $d$ vertices, $P=\{x\in \R^d: x\geq 0, x_i+x_j\leq 1 \;\forall\; i\neq j, i,j\in [d]\}$. Clearly $P$ can be written in form \eqref{eq:2L-embed}, and has $d+{d \choose 2}$ facets. It is not hard to see%
\footnote{We refer to \cite[Chapter 64]{schrijver2003combinatorial} for further details.} 
that $P$ has $d+1$ integral vertices, and exponentially many fractional vertices obtained by setting at least three coordinates to $1/2$ and the others to 0, hence $f_0(P)=2^d-{d \choose 2}$, and we have $f_0(P)f_{d-1}(P)=\left[d+{d \choose 2}\right]\left[2^d-{d \choose 2}\right]> d2^{d+1}$ for $d\geq 5$.

Another natural question is whether the bound of Conjecture~\ref{main-conjecture} holds for integral polytopes that admit a \emph{linear relaxation} of the form \eqref{eq:2L-embed}. 
More formally, let $P_I$ be the integer hull of a polytope $P$. Is it true that, for all $P$ of the form \eqref{eq:2L-embed}, one has $f_0(P_I)f_{d-1}(P_I)\leq d2^{d+1}$? 
Note that this seems to be too general to be true, since such $P$ include, for instance, all stable set polytopes. 
However, given the difficulty of building explicit polytopes with many facets (see~\cite{KRSZ} for some constructions and a discussion), finding a counter-example is non-trivial. 
Through extensive computation with \texttt{polymake}, we found a $12$-dimensional polytope $P$ that violates the conjecture. 
Indeed, for $d=12$ the bound of the conjecture is $98304$, while $f_0(P_I)f_{d-1}(P_I)=535392$. We give an explicit description of the polytope in Appendix~\ref{app:LR}.

\subsection{$0/1$ matrices generalizing slack matrices of $2$-level polytopes}

As mentioned in Section~\ref{sec:intro}, Conjecture~\ref{main-conjecture} can be rephrased as an upper bound on the number of entries of the smallest slack matrices of $2$-level polytopes. It is then a natural question whether one can extend the conjecture on classes of matrices strictly containing those matrices. 

Let $M \in \{0,1\}^{m\times n}$ be a matrix without any repeated row or column. Using the characterization given in~\cite{CharSM}, we have that $M$ is the slack matrix of a $2$-level $d$-polytope $P$ with $d\geq 2$ if and only if: 
\begin{enumerate}[(i)]
	\item $\rk(M)=d+1$;\item The all-ones vector ${\bf 1}$ belongs to the space generated by the rows of $M$; \item The cone generated by the rows of $M$ coincide with the intersection of the space generated by the rows of $M$ with the nonnegative orthant.
\end{enumerate}
Moreover, if $M$ is a minimal slack matrix for $P$, then 
\begin{enumerate}[(iv)]
 \item Rows of $M$ have incomparable supports; and
 \item Columns of $M$ have incomparable supports.
\end{enumerate}
We would like to understand what happens to Conjecture~\ref{main-conjecture} when one of those properties is relaxed. 

Relaxing (i) does not make sense, since it leads to slack matrices of $2$-level polytopes of any dimension, which clearly violate the conjecture. Now suppose we relax (iv). Let $M$ be the slack matrix of the $d$-dimensional cube, and $M'$ obtained from $M$ by adding a row of $1s$. $M'$ verifies properties (i)-(ii)-(iii)-(v), since it is obtained from $M$ by adding a row that is already in the conic hull of rows of $M$. On the other hand, since the cube verifies Conjecture~\ref{main-conjecture} at equality, $M'$ does not verify the conjecture. Similarly, if we relax (v) instead of (iv), add a column of $1$ to $M$ as to obtain $M''$. Note that this new column is also in the conic hull of the columns of $M$, since $M^\intercal$ is the slack matrix of the $d$-dimensional cross-polytope. Hence $M''$ verifies properties (i)-(ii)-(iii)-(iv) but not Conjecture~\ref{main-conjecture}. Finding counterexamples to the conjecture when property (ii) or (iii) are relaxed seems to be harder, hence an interesting open question. Note that, when (ii) is relaxed, the question again has a geometric interpretation, since $M$ is the slack matrix of a polyhedral cone, see~\cite{CharSM}. 

We now investigate what happens if we relax the conditions above even further, and only impose that the rank of $M \in \{0,1\}^{m \times n}$ be $d$, and that $M$ do not have any repeated row or column. From the discussion above, we know that $M$ does not verify Conjecture~\ref{main-conjecture}, but which bound can one give on $m\cdot n?$ An easy argument implies that the maximum number of distinct rows (resp. columns) is $2^d$, hence $m\cdot n\leq 4^d$. Indeed, consider $c_1,\dots,c_d$ linearly independent columns of $M$. Any other column is a linear combinations of the $c_i$'s. But then, if two rows coincide on $c_1,\dots,c_d$, then they are equal, a contradiction. Hence all the rows must be distinct on $c_1,\dots, c_d$, but then, since $M$ has entries 0/1, there can be at most $2^d$ rows. 

We now show that the bound $m\cdot n \leq 4^d$ is not tight. We first show the following:
\begin{lemma}\label{lem:rowandcols} Let $M \in \{0,1\}^{m \times n}$ be of rank $d$, with no repeated rows or columns, and suppose it contains the identity matrix $I_d$ as submatrix. $m \cdot n \leq (d+1)2^d$. \end{lemma} \begin{proof}
	Assume without loss of generality that $I_d$ is exactly the upper $d \times d$ left corner of $M$. Then the first $d$ columns of $M$, denoted by $c_1,\dots, c_d$, are linearly independent, and the first $d$ entries of $c_i$ form the vector $e_i$ for $i=1,\dots,d$. Since $\rk(M)=d$, every other column $c_i$, $i>d$, can be written as $\sum_{j=1}^d \alpha^{(i)}_j c_j$ for some coefficients $\alpha^{(i)}_j$'s. But the first $d$ entries of such $c_i$ are exactly $\alpha^{(i)}_1,\dots,\alpha^{(i)}_d$, hence, as $M$ has 0/1 entries, we have $\alpha^{(i)}_j\in\{0,1\}$ for any $i,j$. Now, consider a graph $G$ with vertex set $[d]$, where node $j$ and node $k$ are adjacent if, for some $i$, we have $\alpha^{(i)}_j=\alpha^{(i)}_k=1$. Clearly each column of $M$ corresponds to a clique of $G$ (including $c_1,\dots, c_d$, which correspond to singletons). Notice also that two columns $c_i ,c_h$ cannot correspond to the same clique, as this would imply that $\alpha^{(i)}=\alpha^{(h)}$, hence that $c_i=c_h$. Now, for any row $r$ of $M$, consider its first $d$ entries. If for some $j<k\leq d$ we have $r_j=r_k=1$, then we cannot have $\alpha^{(i)}_j=\alpha^{(i)}_k=1$ for any column $i$, otherwise the entry of $M$ corresponding to $r$ and $c_i$ would be at least 2, a contradiction. Hence, each row of $M$ corresponds to a stable set of $G$. As before, notice that no two rows can correspond to the same stable set. But then we can apply Corollary~\ref{cor:trade-off} to $G$: defining $\C',{\cal S}'$ for $G$ as in the corollary, we obtain that the size of $M$ is at most $|\C'||{\cal S}'|\leq (d+1)2^d$.
\end{proof}

Now, it easily follows that any 0/1 matrix with rank $d$, and no repeated rows of columns, cannot have $2^d$ rows \emph{and} $2^d$ columns. Assume that $M$ has $2^d$ rows: we will show that it satisfies the hypothesis of the above claim, i.e. that it contains $I_d$ as a submatrix. Let $c_1,\dots,c_d$ linearly independent columns of $M$, and let $M'$ be $M$ restricted to these columns. As argued before, the rows of $M'$  are all different: two rows that coincide in $M'$ yield equal rows in $M$. But then all possible 0/1 vectors must appear as rows of $M'$, in particular $M'$ (hence $M$) contains $I_d$ as a submatrix. In conclusion, the claim implies that $M$ has at most $d+1$ columns, and analogously, if we assume that $M$ has $2^d$ columns, we obtain that $M$ has at most $d+1$ rows, hence the bound $4^d$ cannot be tight. 

One might wonder whether we can apply the above claim to the slack matrix of some interesting 2-level polytopes, to bound its size. However, we now show that the hypotheses of Lemma~\ref{lem:rowandcols} are too strong to be satisfied by any interesting slack matrix. Let $M$ be a minimal 0/1 slack matrix of a polytope $P$ of dimension $d$, hence $\rk(M)=d+1$, and assume that $M$ contains $I_{d+1}$. We claim that $P$ is the $d+1$-dimensional simplex and $M=I_{d+1}$. The argument is similar to the previous one and we only sketch it. Condition (ii) states that the all-ones vector ${\bf 1}$ belongs to the space generated by the rows of $M$. But this space is generated by those rows $r_1,\dots, r_{d+1}$ of $M$ which contain $I_{d+1}$; hence ${\bf 1}=\sum_{i=1}^{d+1} \alpha_i r_i$, which implies similarly as before that $\alpha_i=1$ for $i=1,\dots,{d+1}$. It then follows from the fact that $M$ has all distinct columns that $M$ has exactly ${d+1}$ columns. Hence $P$ is a $d$-dimensional polytope with $d+1$ vertices, i.e., it is a simplex.

\bigskip

\noindent {\bf Acknowledgements.} We thank Jon Lee for introducing us to min up/down polytopes, and pointing us to the paper~\cite{lee2004min}. We are indebted to the referees for their work, that improved the presentation of the paper and suggested research directions. We moreover thank one of them for an argument that significantly simplified the proof of Lemma~\ref{thr:conj-matroid} and Theorem~\ref{thr:matroid}.

\bibliographystyle{plain}

\bibliography{enum-bibliography}{}

\appendix 

\section{Proof of Proposition~\ref{lem:linearindep}}\label{app:linearindep}
	We first prove the following claim: for an edge $mw\in E$ and two rotations $\rho_1, \rho_2\in \Pi$, if $mw\in \rho_1^+ \cap \rho_2^-$, then $\rho_1$ precedes $\rho_2$. 
	Notice first that $\rho_1$ and $\rho_2$ must be in distinct rotations, because the head and the tail of any rotation are always disjoint. 
	Now, consider any $\mu_0-\mu_z$ path $P$ in $H$: we know that each of $\rho_1$ and $\rho_2$ is generated by an arc in $P$ exactly once, by Lemma~\ref{rotation}~(1), 
	and we also know that the happiness of woman $w$ increases monotonously along the path. 
	If $\rho_2$ was generated before $\rho_1$, this would imply that $w$ leaves partner $m$ only to go back to him later on, which violates monotonicity. This proves the claim.
	
	To prove the thesis, consider the linear combination
	\begin{equation} \label{linearindep}
		\sum_{\rho \in \Pi} \lambda_\rho (\chi^{\rho^+} - \chi^{\rho^-}) = 0,
	\end{equation}
	for some coefficients $\lambda_\rho$, and assume by contradiction that not all coefficients are zero.
	Among all rotations $\rho$ with $\lambda_\rho\neq 0$, let $\rho_2$ be a minimal one on the corresponding restriction of the rotation poset, and let $mw$ be an edge in $\rho_2^-$ (such edge exists as no rotation tail can be empty).
	In \cite[Lemma 3.2.1]{gusfield1989stable} it is proved that each edge in $E$ appears in the tail of at most one rotation (as well as in the head of at most one rotation in $\Pi$). 
	Hence, $mw$ appears in no other tail, so for equation \eqref{linearindep} to hold, $mw$ must appear in the head of a distinct rotation $\rho_1$, with $\lambda_{\rho_1} \neq 0$. 
	By the previous claim, $\rho_1$ precedes $\rho_2$, which contradicts the choice of rotation $\rho_2$. This completes the proof.

\section{Proof of Lemma~\ref{obs:2sumpolydescription}}\label{app:2sumpolydescription} Let $Q=B(M_1)\times B(M_2)\cap\{x\in\R^{E_1\uplus E_2}: x_{p_1}+x_{p_2}=1\}$.  
We first claim that $V(Q)=\{(\chi^{B_1},\chi^{B_2}): B_i\in \calB(M_i), i=1,2, p\in B_1\triangle B_2\}$, where $V(P)$ denotes the vertex set of a polytope $P$. 
The ``$\supseteq$'' inclusion is obvious.  
For the opposite inclusion, we just need to prove that the intersection of $B(M_1)\times B(M_2)$ with the hyperplane $H:=\{x\in\R^{E_1\uplus E_2}: x_{p_1}+x_{p_2}=1\}$ does not create any new vertex. 
Suppose that such a vertex $v$ exists: then $v$ is the intersection of $H$ with (the interior of) an edge of $B(M_1)\times B(M_2)$. 
Notice that, using the properties of adjacency of the cartesian product, we can assume without loss of generality that $v=\lambda w+ (1-\lambda) w'$ for some  $0<\lambda<1$, 
where $w=(\chi^{B_1},\chi^{B_2}), w'=(\chi^{B_1},\chi^{B'_2})$ are vertices of $B(M_1)\times B(M_2)$, with $\chi^{B_2}, \chi^{B'_2}$ adjacent vertices of $B(M_2)$. 
In particular we have that $w_{p_1}=w'_{p_1}$, but this is a contradiction since $w,w'$ must be on two different sides of $H$. 
Now, let $E=E_1\cup E_2-p$ be the ground set of $M$, and consider the projection $\varphi: \R^{E_1\uplus E_2}\rightarrow \R^{E}$, 
i.e. such that $\varphi(\vec{1}_e)=\vec{1}_e$ for any $e\in E$, and $\varphi(\vec{1}_{p_1})=\varphi(\vec{1}_{p_2})=\vec{0}$.
From what we just argued it follows that $\varphi(V(Q))=\{\chi^{B_1\cup B_2-p}: B_i\in \calB(M_i), i=1,2, p\in B_1\triangle B_2\}= V(B(M))$ hence, by convexity, $\varphi(Q)=B(M)$. 
We are left to show that $\varphi$ restricted to $Q$ is injective to conclude that $\varphi$ is a bijection from $Q$ to $B(M)$. 
To see this, assume that there are $x,y\in Q$ such that $\varphi(x)=\varphi(y)$, hence $x_e=y_e$ for any $e\in E$. 
But then since $x,y$ satisfy the rank equality of $B(M_1)$, $$x_{p_1}=\rk(M_1)-\sum_{e\in E_1-p} x_e=\
M_1)-\sum_{e\in E_1-p} y_e=y_{p_1},$$ and arguing similarly we get $x_{p_2}=y_{p_2}$, therefore we have $x=y$.

\section{Proof of Theorem~\ref{thm:oxley2006matroid-tree}}\label{app:oxley2006matroid-tree} We proceed by induction on $n=|E(M)|$. For $n=1$, $M$ is 3-connected, $T$ consists of only one vertex and there is nothing to show. For $n>1$: if $M$ is 3-connected, again there is nothing to show. Otherwise, $M=M'\oplus_2 M''$ for some matroids $M', M''$, that are connected (due to Proposition~\ref{obs:2-sum-facts}) and that satisfy $|E(M')|, |E(M'')|<n$. Hence by induction hypothesis the thesis holds for $M'$, $M''$. Let $T', T''$ be their corresponding trees, with vertices labeled by the 3-connected matroids $M_1',\dots,M_{t_1}'$, and $M_1'',\dots,M_{t_2}''$ respectively, edges labeled $e_1',\dots,e_{t_1-1}'$ and $e_1'',\dots,e_{t_2-1}''$ respectively, and let $t=t_1+t_2$. 
By definition of 2-sum there is exactly one element, which we denote by $e_{t-1}$, in $E(M')\cap E(M'')$. By induction we have:
\begin{align*}
E(M)= & E(M')\cup E(M'')\setminus \{e_{t-1}\}\\
= & \left( E(M_1')\cup \cdots \cup E(M_{t_1}')\setminus \{e_1',\dots,e_{t_1-1}'\}\right) \\
& \cup \left(E(M_1'')\cup \cdots \cup E(M_{t_2}'')\setminus \{e_1'',\dots,e_{t_2-1}''\}\right)\setminus \{e_{t-1}\}.
\end{align*}
We can assume without loss of generality that $\{e_1',\dots,e_{t_1-1}'\}\cap E(M'')=\emptyset$ by renaming the elements of $E(M'')$, and similarly we can assume $\{e_1'',\dots,e_{t_2-1}''\}\cap E(M')=\emptyset$. Since $M'$ satisfies properties 1-3, there is exactly one matroid $M_i'$ such that $e_{t-1}\in E(M_i')$, and similarly there is exactly one matroid $M_j''$ such that $e_{t-1}\in E(M_j'')$. Let $T$ be the tree obtained by joining $T', T''$ through the edge $(M_i', M_j'')$. 
Now, it is easy to check that the matroids labeling the vertices of $T$ will satisfy properties 1-3 after an appropriate renaming of the matroids and relabeling of the edges ($M_i'$ will be renamed $M_i$, $M_j''$ $M_{j+t_1}$, and similarly for the elements $e_i', e_j''$).
The statement about the contraction $T/e_1,\dots,e_{t-1}$ follows by induction: one first contracts the edges in $T'$ ($e_1,\dots, e_{t_1-1}$), then the edges in $T''$ ($e_{t_1},\dots, e_{t-2}$), obtaining vertices labeled by $M'$ and $M''$. Then, contracting the edge $e_{t-1}$ joining $M', M''$ one gets $M'\oplus_2 M''=M$. 

\section{Proof of Lemma~\ref{lem:cycle-space-binary}}\label{app:cycle-space-binary} That $\cal C$ is a vector space can be easily verified using the fact that $\cal C$ is closed under taking symmetric difference. This immediately derives from a characterization of binary matroids that can be found in~\cite{oxley2006matroid}, Theorem 9.1.2: $M$ is binary if and only if the symmetric difference of any set of circuits is a disjoint union of circuits. We will now give a basis for $\cal C$ of size $d-r$. The construction is analogous to the construction of a fundamental cycle basis in the cycle space of a graph. Consider a basis $B$ of $M$. For any $e \in E\setminus B$, let $C_e$ denote the unique circuit contained in $B+e$ (note that $e\in C_e$). Since $|B|=r$, we have a family $B_{\cal C} =\{C_{e_1},\dots, C_{e_{d-r}}\}$ of the desired size. Note that the $C_{e_i}$'s are all linearly independent: indeed, $C_e$ cannot be expressed as symmetric difference of other members of $B_{\cal C}$ since it is the only one containing $e$. We are left to show that $B_{\cal C}$ generates $\cal C$. Let $C \in {\cal C}$, $C\neq \emptyset$, and let $\{e_1,\dots, e_{d-r}\}\cap C=e_{i_1},\dots, e_{i_k}$ for some $k\geq 1$ (indeed, $C\not\subseteq B$). Consider now $D=C\triangle C_{e_{i_1}}\triangle\dots\triangle C_{e_{i_k}}$. $D$ is a cycle, however one can see that it is contained in $B$: for each $e\in E\setminus B$, if $e \in C$ then $e$ appears exactly twice in the expression of $D$, hence $e\not\in D$; if $e\not\in C$, $e$ does not appear in the expression at all. This implies that $D=\emptyset$, which is equivalent to $C=C_{e_{i_1}}\triangle\dots\triangle C_{e_{i_k}}$.

\section{The polytopes from Proposition~\ref{lem:notHansen}}\label{app:notHansen} The following are the \texttt{polymake} vertex descriptions of the two $8$-dimensional polytopes from Proposition~\ref{lem:notHansen}: 
the min up/down polytope $P_8(2)$ is denoted by \$P, and the Hansen polytope $\Hans(P_7)$ of the path on 7 nodes $P_7$ is denoted by \$H. 

\smallskip

\texttt{\$P = new Polytope(VERTICES=> $[$ \\ 
	$[$1, 0, 0, 0, 0, 0, 0, 0, 0$], [$1, 1, 1, 1, 1, 1, 1, 1, 1$],$\\ 
	$[$1, 0, 0, 0, 0, 0, 0, 0, 1$], [$1, 1, 1, 1, 1, 1, 1, 1, 0$],$ \\
	$[$1, 0, 0, 0, 0, 0, 0, 1, 1$], [$1, 1, 1, 1, 1, 1, 1, 0, 0$],$ \\
	$[$1, 0, 0, 0, 0, 0, 1, 1, 1$], [$1, 1, 1, 1, 1, 1, 0, 0, 0$],$ \\
	$[$1, 0, 0, 0, 0, 0, 1, 1, 0$], [$1, 1, 1, 1, 1, 1, 0, 0, 1$],$ \\
	$[$1, 0, 0, 0, 0, 1, 1, 1, 1$], [$1, 1, 1, 1, 1, 0, 0, 0, 0$],$ \\
	$[$1, 0, 0, 0, 0, 1, 1, 1, 0$], [$1, 1, 1, 1, 1, 0, 0, 0, 1$],$ \\
	$[$1, 0, 0, 0, 0, 1, 1, 0, 0$], [$1, 1, 1, 1, 1, 0, 0, 1, 1$],$ \\
	$[$1, 0, 0, 0, 1, 1, 1, 1, 1$], [$1, 1, 1, 1, 0, 0, 0, 0, 0$],$ \\
	$[$1, 0, 0, 0, 1, 1, 1, 1, 0$], [$1, 1, 1, 1, 0, 0, 0, 0, 1$],$ \\
	$[$1, 0, 0, 0, 1, 1, 1, 0, 0$], [$1, 1, 1, 1, 0, 0, 0, 1, 1$],$ \\
	$[$1, 0, 0, 0, 1, 1, 0, 0, 0$], [$1, 1, 1, 1, 0, 0, 1, 1, 1$],$ \\
	$[$1, 0, 0, 0, 1, 1, 0, 0, 1$], [$1, 1, 1, 1, 0, 0, 1, 1, 0$],$ \\
	$[$1, 0, 0, 1, 1, 1, 1, 1, 1$], [$1, 1, 1, 0, 0, 0, 0, 0, 0$],$ \\
	$[$1, 0, 0, 1, 1, 1, 1, 1, 0$], [$1, 1, 1, 0, 0, 0, 0, 0, 1$],$ \\
	$[$1, 0, 0, 1, 1, 1, 1, 0, 0$], [$1, 1, 1, 0, 0, 0, 0, 1, 1$],$ \\
	$[$1, 0, 0, 1, 1, 1, 0, 0, 0$], [$1, 1, 1, 0, 0, 0, 1, 1, 1$],$ \\
	$[$1, 0, 0, 1, 1, 1, 0, 0, 1$], [$1, 1, 1, 0, 0, 0, 1, 1, 0$],$ \\
	$[$1, 0, 0, 1, 1, 0, 0, 0, 0$], [$1, 1, 1, 0, 0, 1, 1, 1, 1$],$ \\
	$[$1, 0, 0, 1, 1, 0, 0, 0, 1$], [$1, 1, 1, 0, 0, 1, 1, 1, 0$],$ \\
	$[$1, 0, 0, 1, 1, 0, 0, 1, 1$], [$1, 1, 1, 0, 0, 1, 1, 0, 0$],$ \\
	$[$1, 0, 1, 1, 1, 1, 1, 1, 1$], [$1, 1, 0, 0, 0, 0, 0, 0, 0$],$ \\
	$[$1, 0, 1, 1, 1, 1, 1, 1, 0$], [$1, 1, 0, 0, 0, 0, 0, 0, 1$],$ \\
	$[$1, 0, 1, 1, 1, 1, 1, 0, 0$], [$1, 1, 0, 0, 0, 0, 0, 1, 1$],$ \\
	$[$1, 0, 1, 1, 1, 1, 0, 0, 0$], [$1, 1, 0, 0, 0, 0, 1, 1, 1$],$ \\
	$[$1, 0, 1, 1, 1, 1, 0, 0, 1$], [$1, 1, 0, 0, 0, 0, 1, 1, 0$],$ \\
	$[$1, 0, 1, 1, 1, 0, 0, 0, 0$], [$1, 1, 0, 0, 0, 1, 1, 1, 1$],$ \\
	$[$1, 0, 1, 1, 1, 0, 0, 0, 1$], [$1, 1, 0, 0, 0, 1, 1, 1, 0$],$ \\
	$[$1, 0, 1, 1, 1, 0, 0, 1, 1$], [$1, 1, 0, 0, 0, 1, 1, 0, 0$],$ \\
	$[$1, 0, 1, 1, 0, 0, 0, 0, 0$], [$1, 1, 0, 0, 1, 1, 1, 1, 1$],$ \\
	$[$1, 0, 1, 1, 0, 0, 0, 0, 1$], [$1, 1, 0, 0, 1, 1, 1, 1, 0$],$ \\
	$[$1, 0, 1, 1, 0, 0, 0, 1, 1$], [$1, 1, 0, 0, 1, 1, 1, 0, 0$],$ \\
	$[$1, 0, 1, 1, 0, 0, 1, 1, 1$], [$1, 1, 0, 0, 1, 1, 0, 0, 0$],$ \\
	$[$1, 0, 1, 1, 0, 0, 1, 1, 0$], [$1, 1, 0, 0, 1, 1, 0, 0, 1$]]$ );}

\texttt{\$H= new Polytope(VERTICES=> $[$ \\
	$[$1, 1, 0, 0, 0, 0, 0, 0, 0$], [$1, -1, 0, 0, 0, 0, 0, 0, 0$],$ \\
	$[$1, 1, 0, 0, 0, 0, 0, 0, 1$], [$1, -1, 0, 0, 0, 0, 0, 0, -1$],$ \\
	$[$1, 1, 0, 0, 0, 0, 0, 1, 0$], [$1, -1, 0, 0, 0, 0, 0, -1, 0$],$ \\
	$[$1, 1, 0, 0, 0, 0, 1, 0, 0$], [$1, -1, 0, 0, 0, 0, -1, 0, 0$],$ \\
	$[$1, 1, 0, 0, 0, 0, 1, 0, 1$], [$1, -1, 0, 0, 0, 0, -1, 0, -1$],$ \\
	$[$1, 1, 0, 0, 0, 1, 0, 0, 0$], [$1, -1, 0, 0, 0, -1, 0, 0, 0$],$ \\
	$[$1, 1, 0, 0, 0, 1, 0, 0, 1$], [$1, -1, 0, 0, 0, -1, 0, 0, -1$],$ \\
	$[$1, 1, 0, 0, 0, 1, 0, 1, 0$], [$1, -1, 0, 0, 0, -1, 0, -1, 0$],$ \\
	$[$1, 1, 0, 0, 1, 0, 0, 0, 0$], [$1, -1, 0, 0, -1, 0, 0, 0, 0$],$ \\
	$[$1, 1, 0, 0, 1, 0, 0, 0, 1$], [$1, -1, 0, 0, -1, 0, 0, 0, -1$],$ \\
	$[$1, 1, 0, 0, 1, 0, 0, 1, 0$], [$1, -1, 0, 0, -1, 0, 0, -1, 0$],$ \\
	$[$1, 1, 0, 0, 1, 0, 1, 0, 0$], [$1, -1, 0, 0, -1, 0, -1, 0, 0$],$ \\
	$[$1, 1, 0, 0, 1, 0, 1, 0, 1$], [$1, -1, 0, 0, -1, 0, -1, 0, -1$],$ \\
	$[$1, 1, 0, 1, 0, 0, 0, 0, 0$], [$1, -1, 0, -1, 0, 0, 0, 0, 0$],$ \\
	$[$1, 1, 0, 1, 0, 0, 0, 0, 1$], [$1, -1, 0, -1, 0, 0, 0, 0, -1$],$ \\
	$[$1, 1, 0, 1, 0, 0, 0, 1, 0$], [$1, -1, 0, -1, 0, 0, 0, -1, 0$],$ \\
	$[$1, 1, 0, 1, 0, 0, 1, 0, 0$], [$1, -1, 0, -1, 0, 0, -1, 0, 0$],$ \\
	$[$1, 1, 0, 1, 0, 0, 1, 0, 1$], [$1, -1, 0, -1, 0, 0, -1, 0, -1$],$ \\
	$[$1, 1, 0, 1, 0, 1, 0, 0, 0$], [$1, -1, 0, -1, 0, -1, 0, 0, 0$],$ \\
	$[$1, 1, 0, 1, 0, 1, 0, 0, 1$], [$1, -1, 0, -1, 0, -1, 0, 0, -1$],$ \\
	$[$1, 1, 0, 1, 0, 1, 0, 1, 0$], [$1, -1, 0, -1, 0, -1, 0, -1, 0$],$ \\
	$[$1, 1, 1, 0, 0, 0, 0, 0, 0$], [$1, -1, -1, 0, 0, 0, 0, 0, 0$],$ \\
	$[$1, 1, 1, 0, 0, 0, 0, 0, 1$], [$1, -1, -1, 0, 0, 0, 0, 0, -1$],$ \\
	$[$1, 1, 1, 0, 0, 0, 0, 1, 0$], [$1, -1, -1, 0, 0, 0, 0, -1, 0$],$ \\
	$[$1, 1, 1, 0, 0, 0, 1, 0, 0$], [$1, -1, -1, 0, 0, 0, -1, 0, 0$],$ \\
	$[$1, 1, 1, 0, 0, 0, 1, 0, 1$], [$1, -1, -1, 0, 0, 0, -1, 0, -1$],$ \\
	$[$1, 1, 1, 0, 0, 1, 0, 0, 0$], [$1, -1, -1, 0, 0, -1, 0, 0, 0$],$ \\
	$[$1, 1, 1, 0, 0, 1, 0, 0, 1$], [$1, -1, -1, 0, 0, -1, 0, 0, -1$],$ \\
	$[$1, 1, 1, 0, 0, 1, 0, 1, 0$], [$1, -1, -1, 0, 0, -1, 0, -1, 0$],$ \\
	$[$1, 1, 1, 0, 1, 0, 0, 0, 0$], [$1, -1, -1, 0, -1, 0, 0, 0, 0$],$ \\
	$[$1, 1, 1, 0, 1, 0, 0, 0, 1$], [$1, -1, -1, 0, -1, 0, 0, 0, -1$],$ \\
	$[$1, 1, 1, 0, 1, 0, 0, 1, 0$], [$1, -1, -1, 0, -1, 0, 0, -1, 0$],$ \\
	$[$1, 1, 1, 0, 1, 0, 1, 0, 0$], [$1, -1, -1, 0, -1, 0, -1, 0, 0$],$ \\
	$[$1, 1, 1, 0, 1, 0, 1, 0, 1$], [$1, -1, -1, 0, -1, 0, -1, 0, -1$]]$ );}

\section{The polytope from Section~\ref{ex:LR}}\label{app:LR} The following is the \texttt{polymake} inequality description for the $12$-dimensional polytope from Example~\ref{ex:LR}.

\smallskip

\texttt{\$c= new Polytope(INEQUALITIES=>$[$\\
	$[$0, 1, 1, 1, 1, 1, 1, 1, 1, 1, 0, 1, 1$]$
	,\\
	$[$1, -1, -1, -1, -1, -1, -1, -1, -1, -1, 0, -1, -1$]$
	,\\
	$[$0, 1, 0, 1, 0, 1, 1, 1, 0, 1, 1, 0, 1$]$
	,\\
	$[$1, -1, 0, -1, 0, -1, -1, -1, 0, -1, -1, 0, -1$]$
	,\\
	$[$0, 0, 0, 0, 0, 1, 0, 0, 0, 1, 1, 0, 1$]$
	,\\
	$[$1, 0, 0, 0, 0, -1, 0, 0, 0, -1, -1, 0, -1$]$
	,\\
	$[$0, 1, 0, 0, 0, 0, 0, 0, 1, 0, 0, 0, 0$]$
	,\\
	$[$1, -1, 0, 0, 0, 0, 0, 0, -1, 0, 0, 0, 0$]$
	,\\
	$[$0, 1, 1, 1, 1, 1, 1, 1, 1, 1, 1, 1, 1$]$
	,\\
	$[$1, -1, -1, -1, -1, -1, -1, -1, -1, -1, -1, -1, -1$]$
	,\\
	$[$0, 1, 1, 1, 1, 1, 0, 0, 0, 0, 0, 1, 1$]$
	,\\
	$[$1, -1, -1, -1, -1, -1, 0, 0, 0, 0, 0, -1, -1$]$
	,\\
	$[$0, 1, 0, 0, 0, 0, 0, 0, 0, 0, 1, 0, 0$]$
	,\\
	$[$1, -1, 0, 0, 0, 0, 0, 0, 0, 0, -1, 0, 0$]$
	,\\
	$[$0, 1, 0, 0, 0, 0, 0, 0, 1, 0, 0, 0, 1$]$
	,\\
	$[$1, -1, 0, 0, 0, 0, 0, 0, -1, 0, 0, 0, -1$]$
	,\\
	$[$0, 1, 1, 0, 1, 0, 0, 0, 0, 0, 1, 0, 0$]$
	,\\
	$[$1, -1, -1, 0, -1, 0, 0, 0, 0, 0, -1, 0, 0$]$
	,\\
	$[$0, 1, 1, 1, 0, 1, 1, 0, 1, 1, 0, 1, 0$]$
	,\\
	$[$1, -1, -1, -1, 0, -1, -1, 0, -1, -1, 0, -1, 0$]$
	,\\
	$[$0, 0, 1, 0, 0, 0, 0, 1, 0, 0, 0, 0, 1$]$
	,\\
	$[$1, 0, -1, 0, 0, 0, 0, -1, 0, 0, 0, 0, -1$]$
	,\\
	$[$0, 1, 1, 1, 1, 1, 1, 1, 1, 1, 0, 1, 1$]$
	,\\
	$[$1, -1, -1, -1, -1, -1, -1, -1, -1, -1, 0, -1, -1$]$
	,\\
	$[$0, 1, 0, 0, 0, 0, 0, 0, 0, 0, 0, 0, 0$]$
	,\\
	$[$1, -1, 0, 0, 0, 0, 0, 0, 0, 0, 0, 0, 0$]$
	,\\
	$[$0, 1, 1, 1, 0, 0, 0, 0, 0, 1, 1, 0, 0$]$
	,\\
	$[$1, -1, -1, -1, 0, 0, 0, 0, 0, -1, -1, 0, 0$]$
	,\\
	$[$0, 1, 1, 1, 1, 1, 1, 1, 1, 1, 1, 1, 1$]$
	,\\
	$[$1, -1, -1, -1, -1, -1, -1, -1, -1, -1, -1, -1, -1$]$
	,\\
	$[$0, 1, 1, 1, 0, 0, 1, 1, 1, 1, 0, 1, 1$]$
	,\\
	$[$1, -1, -1, -1, 0, 0, -1, -1, -1, -1, 0, -1, -1$]$
	$]$);}


\end{document}